\documentclass[a4paper,12pt]{amsart}
\usepackage{fullpage}
\usepackage[english,francais]{babel}
\usepackage{epsf}
\usepackage[dvips]{epsfig}
\usepackage{amssymb}
\usepackage{amsmath}
\usepackage[all]{xy}
\usepackage{color}

\usepackage{ucs}
\usepackage[utf8x]{inputenc}\usepackage{wasysym}
\usepackage{aeguill}

\newtheorem{theorem}{Theorem}[section]
\newtheorem{prop-f}[theoreme]{Proposition}
\newtheorem{prop}[theorem]{Proposition}

\newtheorem{corollary}[theorem]{Corollary}

\newtheorem{lemma}[theorem]{Lemma}

\newcommand{\finpreuve}{\hspace{\stretch{1}}{$\square$}}

\newcommand{\R}{\mathbb{R}}

\newcommand{\Z}{\mathbb{Z}}
\newcommand{\N}{\mathbb{N}}

\def\esp{\medskip\noindent}
\def\prg#1{\esp{\bf #1. }}

\def\proof{\prg{Proof}}

\def\proofof#1{\prg{Proof of #1}}
\renewcommand{\epsilon}{\varepsilon}
\def\btab{\begin{eqnarray*}}
\def\etab{\end{eqnarray*}}
\def\beq{\begin{equation}}
\def\eeq{\end{equation}}

\newcounter{numeroexo}

\def\L{{\mathcal D}}

\begin{document}

\selectlanguage{english}

\title[Continuum percolation in high dimensions]{Continuum percolation in high dimensions}

{
\author{Jean-Baptiste Gou\'er\'e}
\author{R{\'e}gine Marchand}
\address{Laboratoire de Math{\'e}matiques, Applications et Physique
Math{\'e}matique d'Orl{\'e}ans UMR 6628\\ Universit{\'e} d'Orl{\'e}ans\\ B.P.
6759\\
 45067 Orl{\'e}ans Cedex 2 France}
\email{Jean-Baptiste.Gouere@univ-orleans.fr}

\address{1. Universit{\'e} de Lorraine\\
Institut Elie Cartan de Lorraine, UMR 7502 (math{\'e}matiques)\\
Vandoeuvre-l{\`e}s-Nancy, F-54506 France. 
2. CNRS \\
Institut Elie Cartan de Lorraine, UMR 7502 (math{\'e}matiques)\\
Vandoeuvre-l{\`e}s-Nancy, F-54506 France. }
\email{Regine.Marchand@univ-lorraine.fr}
}


\begin{abstract}
Consider a Boolean model $\Sigma$ in $\R^d$.
The centers are given by a homogeneous Poisson point process with intensity $\lambda$ and the radii of distinct balls are i.i.d.\ with common distribution $\nu$.
The critical covered volume is the proportion of space covered by $\Sigma$ when the intensity $\lambda$ is critical for percolation.
Previous numerical simulations and heuristic arguments suggest that the critical covered volume
may be minimal when $\nu$ is a Dirac measure. 

In this paper, we prove that it is not the case at least in high dimension.
To establish this result we study the asymptotic behaviour, as $d$ tends to infinity, of the critical covered volume.
It appears that, in contrast to what happens in the constant radii case studied by Penrose, geometrical dependencies do not always vanish in high dimension.
\end{abstract}

\maketitle

\section{Introduction and statement of the main results}

\subsection*{Introduction.}
Consider a homogeneous Poisson point process on $\R^d$.
At each point of this process, we center a ball with random radius, the radii of distinct balls being i.i.d.\
and independent of the point process.
The union $\Sigma$ of these random balls is called a Boolean model.
This Boolean model only depends on three parameters : the intensity $\lambda$ of the point process of centers, the common distribution $\nu$ of the radii of the balls and the dimension $d$.

Denote by $\lambda_d^c(\nu)$ the critical intensity for percolation in $\Sigma$.
We then consider $c^c_d(\nu)$, the volumic proportion of space which is covered by $\Sigma$ when $\lambda=\lambda_d^c(\nu)$.
This quantity is called the critical covered volume.
This quantity is scale invariant (see \eqref{e:scaling}). For example $c_d^c(\delta_1)=c_d^c(\delta_r)$ for any $r>0$, where $\delta_r$ denotes the Dirac measure on $r$.
Numerical simulations in low dimension and heuristic arguments in any dimension suggested that 
the critical covered volume may be minimal when $\nu$ is a Dirac measure, that is when all the balls have the same radius. 
We show that this is not true in high dimensions.
This result is proved through the study of the following kind of asymptotics.
Let $\mu$ be a probability measure on $]0,+\infty[$ such that $r^{-d}\mu(dr)$ is a finite measure for any $d \ge 1$.
For any $d$, we then consider the probability measure $\widetilde{\mu_d} = Z_d^{-1} r^{-d} \mu(dr)$ where $Z_d$ is a normalization constant.
The $r^{-d}$ normalization will be discussed and motivated below \eqref{e:mud}. 
We prove that, as soon as $\mu$ is non degenerate, one has :
\begin{equation}\label{e:ll}
c_d^c(\widetilde{\mu_d}) \ll c_d^c(\delta_1) \hbox{ as } d\to\infty.
\end{equation}
This proves that, when $d$ is large enough, the critical covered volume is not minimal in the constant radii case.

The constant radii case has been studied by Penrose \cite{Penrose-high-dimensions}.
He proved that the asymptotic behavior of $c_d^c(\delta_1)$ is given by the critical parameter of the associated Galton-Watson process. 
This is due to the fact that geometrical dependencies vanish in high dimension in that case. 
We first focus in this paper on the case where the radii take only two distinct values.
We prove that, in that case, the asymptotic behavior of $c_d^c(\widetilde{\mu_d})$ is given by a competition between genealogy effects 
(given by the associated two-type Galton-Watson process) and geometrical dependencies effects.
This yields \eqref{e:ll} in that case.
The general non degenerate distribution of radii case follows.

\subsection*{The Boolean model.}
Let us give here a different -- but equivalent -- construction of the Boolean model.
Let $\nu$ be a finite\footnote{There is no greater generality in considering finite measures instead of
probability measures ; this is simply more convenient.} measure on $(0,+\infty)$.
We assume that the mass of $\nu$ is positive. 
Let $d \ge 2$ be an integer, $\lambda > 0$ be a real number and  
$\xi$ be a Poisson point process on $\R^d \times (0,+\infty)$ 
whose intensity measure is the Lebesgue measure on $\R^d$ times $\lambda\nu$.
We define a random subset of $\R^d$ as follows:
$$
\Sigma(\lambda\nu) = \bigcup_{(c,r) \in \xi} B(c,r),
$$
where $B(c,r)$ is the open Euclidean ball centered at $c \in \R^d$ and with radius $r \in (0,+\infty)$.
The random subset $\Sigma(\lambda\nu)$ is a Boolean model driven by $\lambda\nu$.


We say that $\Sigma(\lambda\nu)$ percolates if 
the connected component of $\Sigma(\lambda\nu)$ that contains the origin is unbounded with positive probability.
This is equivalent to the almost-sure existence of an unbounded connected component of $\Sigma(\lambda\nu)$.
We refer to the book by Meester and Roy \cite{Meester-Roy-livre} for background on continuum percolation.
The critical intensity is defined by:
$$
\lambda^c_d(\nu) = \inf \{\lambda>0 : \Sigma(\lambda\nu) \hbox{ percolates}\}.
$$
One easily checks that $\lambda^c_d(\nu)$ is finite.
In \cite{G-perco-boolean-model} it is proven that $\lambda^c_d(\nu)$ is positive if and only~if 
\begin{equation}\label{e:volumefini}
\int r^d\nu(dr) < +\infty.
\end{equation}
We assume that this assumption is fulfilled.

By ergodicity, the Boolean model $\Sigma(\lambda\nu)$ has a deterministic natural density.
This is also the probability that a given point belongs to the Boolean model and  it is given by :
$$
P(0 \in \Sigma(\lambda\nu))=1-\exp\left(-\lambda \int v_dr^d \nu(dr)\right),
$$
where $v_d$ denotes the volume of the unit ball in $\R^d$.
The critical covered volume $c_d^c(\nu)$ is the density of the Boolean model when the intensity is critical :
$$
c^c_d(\nu) =1-\exp\left(-\lambda_d^c(\nu) \int v_dr^d \nu(dr) \right).
$$
It is thus more convenient to study the critical covered volume through the normalized critical intensity:
$$
\widetilde{\lambda}^c_d(\nu) = \lambda_d^c(\nu) \int v_d(2r)^d \nu(dr).
$$
We then have ${c^c_d}(\nu) =1-\exp\left(-\frac{\widetilde\lambda_d^c(\nu)}{2^d} \right)$.
The factor $2^d$ may seem arbitrary here.
Its interest will appear in the statement of the next theorems.

We will now give two scaling relations which partly justify our preference for $c^c_d$ or $\widetilde{\lambda}^c_d$ over $\lambda_d^c$.
For all $a>0$, define $H^a\nu$ as the image of $\nu$ under the map defined by $x \mapsto ax$.
By scaling, we get:
\begin{equation}\label{e:scaling}
\widetilde{\lambda}^c_d(H^a\nu)=\widetilde{\lambda}^c_d(\nu).
\end{equation}
This is a consequence of Proposition 2.11 in \cite{Meester-Roy-livre}, and it may become more obvious when considering
the two following facts : a critical Boolean model remains critical when rescaling and 
the density is invariant by rescaling ; therefore the critical covered volume and then the normalized threshold are invariant.
One also easily checks the following invariance:
\begin{equation}\label{e:scaling2}
\widetilde{\lambda}^c_d(a\nu)=\widetilde{\lambda}^c_d(\nu).
\end{equation}

\subsection*{Critical intensity as a function of $\nu$.}
It has been conjectured by Kertész and Vicsek~\cite{Kertesz-al} that the critical covered volume -- or equivalently the normalized critical intensity -- should be independent of $\nu$,
as soon as the support of $\nu$ is bounded.
Phani and Dhar \cite{Phani-Dhar} gave a heuristic argument suggesting that the conjecture were false.
A rigorous proof was then given by Meester, Roy and Sarkar in \cite{Meester-Roy-Sarkar}.
More precisely, they gave examples of measures $\nu$ with two atoms such that:
\begin{equation}\label{e:MRS}
\widetilde\lambda^c_d(\nu) > \widetilde\lambda_d^c(\delta_1).
\end{equation}
As a consequence of Theorem 1.1 in the paper by Menshikov, Popov and Vachkovskaia~\cite{Menshikov-al-multi},
we even get that $\widetilde\lambda^c_d(\nu)$ can be arbitrarily large.
More precisely, if
\begin{equation}\label{e:cvmulti}
\text{if }\nu(n,a)=\sum_{k=0}^{n-1} a^{dk} \delta_{a^{-k}}, \text{ then }\widetilde\lambda^c_d(\nu(n,a)) \to n\widetilde\lambda^c_d(\delta_1) \hbox{ as } a \to \infty.
\end{equation}
Actually the result of \cite{Menshikov-al-multi} is the following much stronger statement: 
$\lambda^c_d(\nu(+\infty,a)) \to \lambda_d^c(\delta_1)$ when $a \to \infty$.
The convergence \eqref{e:cvmulti} is implicit in the work of Meester, Roy and Sarkar in \cite{Meester-Roy-Sarkar}, 
at least when $n=2$.
There were also heuristics for such a result in \cite{Phani-Dhar}.

By Theorem 2.1 in \cite{G-perco-boolean-model}, we get the existence of a positive constant $C_d$, that depends only on the 
dimension $d$, such that:
$$
\widetilde\lambda^c_d(\nu) \ge C_d.
$$

To sum up, $\widetilde\lambda_d^c(\cdot)$ is not bounded from above but is bounded from below by a positive constant.
In other words, the critical covered volume $c^c_d(\cdot) \in (0,1)$ can be arbitrarily close to $1$ 
but is bounded from below by a positive constant. It is thus natural to seek the optimal distribution, that is the one which minimizes the critical covered volume.

In the physical literature, it is strongly believed that, at least when $d=2$ and $d=3$, the critical covered volume is minimum in the case of a deterministic radius, when the distribution of radius is a Dirac measure.
This conjecture is supported by numerical evidence (to the best of our knowledge, the most accurate estimations are given in
a paper by Quintanilla and Ziff \cite{QZ-PRE-2007} when $d=2$ and in
a paper by Consiglio, Baker, Paul and Stanley \cite{Consiglio-2003} when $d=3$). On Figure \ref{f:multiscale}, we plot the critical covered volume in dimension 2 as a function of $\alpha$ and for different values of $\rho$ when $\nu=(1-\alpha) \delta_1+{\alpha}{\rho^{-2}}\delta_\rho$. The data for finite values of $\rho$ come from numerical estimations in \cite{QZ-PRE-2007}, while the data for infinite $\rho$ come from the study of the multi-scale Boolean model. See Section 1.4 in \cite{G-multi} for further references. 
\begin{figure}[h!]
\includegraphics[scale=0.7]{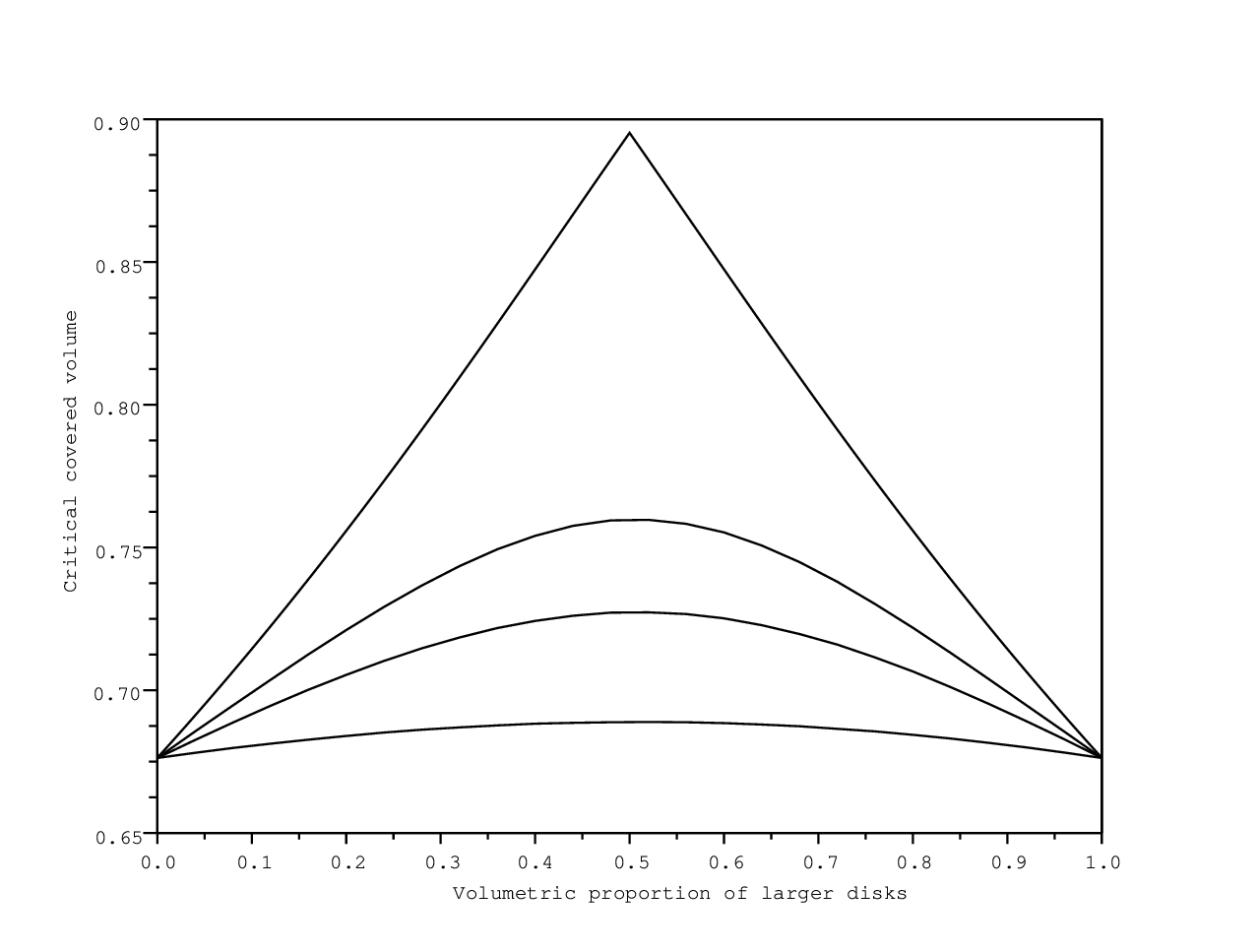}
\caption{Critical covered volume as a function of $\alpha$ for different values of $\rho$.
From bottom to top: $\rho=2, \rho=5, \rho=10$ and the limit as $\rho \to \infty$.}
\label{f:multiscale}
\end{figure}
The conjecture is also supported by some heuristic arguments in any dimension (see for example Dhar~\cite{Dhar-1997}).
See also \cite{Balram-Dhar}.
In the above cited paper \cite{Meester-Roy-Sarkar}, 
it is noted that the rigorous proof of \eqref{e:MRS} suggests that the deterministic case might be optimal for any $d \ge 2$.

In this paper we show on the contrary that for all $d$ large enough the critical covered volume is not minimized by the case of deterministic radii. 

\subsection*{Critical intensity  in high dimension : the case of a deterministic radius.}

Assume here that the measure $\nu$ is a Dirac mass at $1$, that is that the radii of the balls are all equal to $1$. 
Penrose proved the following result in \cite{Penrose-high-dimensions} :
\begin{theorem}[Penrose] \label{t:penrose}
$\displaystyle\lim_{d\to\infty} \widetilde\lambda^c_d(\delta_1) =1.$
\end{theorem}
With the scale invariance  \eqref{e:scaling} of $\widetilde\lambda^c_d$, this limit can readily be generalized to any constant radius : for any $a>0$,
$$
\lim_{d\to\infty} \widetilde\lambda^c_d(\delta_a) = 
\lim_{d\to\infty} \widetilde\lambda^c_d(\delta_1) =1.
$$   
Theorem \ref{t:penrose} is the continuum analogue of a result of Kesten
\cite{Kesten-high-dimensions} 
for Bernoulli bond percolation on the nearest-neighbor integer lattice $\Z^d$, which says that the critical percolation parameter is asymptotically equivalent to $1/(2d)$.

Let us say a word about the ideas of the proof of Theorem \ref{t:penrose}.

The inequality $\widetilde\lambda^c_d(\delta_1) > 1$  
holds  for any $d \ge 2$.
The proof is simple, and here is the idea. We consider the following natural genealogy.
The deterministic ball $B(0,1)$ is said to be the ball of generation $0$.
The random balls of $\Sigma(\lambda\delta_1)$ that touch $B(0,1)$ are then the balls of generation $1$.
The random balls that touch one ball of generation $1$ without being one of them are then  the balls of generation $2$ and so on.
Let us denote by $N_d$ the number of all balls that are descendants of $B(0,1)$.
There is no percolation if and only if $N_d$ is almost surely finite.

Now denote by $m$ the Poisson distribution with mean $\lambda v_d 2^d $ : this is the law of the
number of balls of $\Sigma(\lambda \delta_1)$ that touch a given ball of radius $1$.
Therefore, if there were no interference between children of different balls, $N_d$ would be equal to $Z$, 
the total population in a Galton-Watson process with offspring distribution $m$.
Because of the interferences due to the fact that the Boolean model lives in $\R^d$, this is not true : in fact, $N_d$ is only stochastically dominated by $Z$.
Therefore, if $\lambda v_d2^d \le 1$, then $Z$ is finite almost surely, then  $N_d$ is finite almost surely
and therefore there is no percolation.
This implies 
$$\widetilde\lambda_d^c(\delta_1) = v_d2^d \lambda_d^c(\delta_1) > 1.$$
The difficult part of Theorem \ref{t:penrose} is to prove that if $d$ is large, then the interferences are small,
then $N_d$ is close to $Z$ and therefore there is percolation for large $d$ as soon as 
$v_d2^d \lambda$ is a constant striclty larger than one.

To sum up, at first order, the asymptotic behavior of the critical intensity of the Boolean model with constant radius is given by 
the threshold of the associated Galton-Watson process, as in the case of Bernoulli percolation on $\Z^d$ : roughly speaking, as the dimension increases, the geometrical constraints of the finite dimension space decrease and at the limit, we recover the non-geometrical case of the corresponding Galton-Watson process.

\subsection*{Critical intensity  in high dimension : the case of random radii.}

If $\mu$ is a finite measure on $(0,+\infty)$ and if $d\ge 2$ is an integer, we define a measure $\mu_d$ on $(0,+\infty)$ by setting :
\begin{equation}
\label{e:mud}
\mu_d(dr) = r^{-d} \mu(dr).
\end{equation}
Note that, for any $d$, the assumption \eqref{e:volumefini} is fulfilled by $\mu_d$, and that $(\delta_1)_d=\delta_1$.
Note also that $\mu_d$ is not necessarily a finite measure.
However the definitions made above still make sense in this case and we still have $\lambda_d^c(\mu_d) \in (0,+\infty)$
thanks to Theorem 1.1 in \cite{G-perco-generale}.

We will study the behavior of $\widetilde\lambda_d^c(\mu_d)$ as $d$ tends to infinity.
Let us motivate the definition of $\mu_d$ with the following two related properties :
\begin{enumerate}
\item
Consider the Boolean model $\Sigma(\lambda\mu_d)$ on $\R^d$ driven by $\lambda\mu_d$ where $\lambda>0$.
For any $0<s<t<\infty$, the number of balls of $\Sigma(\lambda\mu_d)$ with radius in $[s,t]$ that contains a given point
is a Poisson random variable with intensity:
$$
\int_{[s,t]} v_dr^d \lambda\mu_d(dr) = v_d \lambda\mu([s,t]). 
$$
Loosely speaking, this means that contrary to what happens in the Boolean model driven by $\lambda \mu$, the relative importance of radii of different sizes does not depend on the dimension $d$ in the Boolean model driven by $\lambda \mu_d$.
\item
A closely related property is the following one.
Consider for example the case $\mu=\alpha \delta_a + \beta \delta_b$.
Then, a way to build $\Sigma(\lambda\mu_d)$ is to proceed as follows.
Consider two independent Boolean model: 
$\Sigma^A$, driven by $\lambda\alpha \delta_1$, and $\Sigma^B$, driven by $\lambda\beta \delta_1$.
Then set $\Sigma(\lambda\mu_d) = a \Sigma^A + b \Sigma^B$.
\end{enumerate}
We prove the following result :
\begin{theorem}  \label{t:general}	
Let $\mu$ be a finite measure on $(0,+\infty)$. 
We assume that the mass of $\mu$ is positive and that $\mu$ is not concentrated on a singleton.
Then :
$$
\limsup_{d \to + \infty} \frac1d \ln\left(\widetilde{\lambda}^c_d(\mu_d)\right) < 0.
$$
\end{theorem}

As $(\delta_1)_d=\delta_1$, a straightforward consequence of Theorem \ref{t:general} and Theorem \ref{t:penrose} 
-- or, actually, of the much weaker and easier convergence of $\ln(\widetilde{\lambda}_d^c(\delta_1))$ to $0$ --
 is the following result:

\begin{corollary} \label{t:conjecture}
Let $\mu$ be a finite mesure on $(0,+\infty)$.
We assume that the mass of $\mu$ is positive and that $\mu$ is not concentrated on a singleton.
Then, for any $d$ large enough, we have:
$$
\widetilde{\lambda}_d^c(\mu_d) < \widetilde{\lambda}_d^c(\delta_1), \text{ or equivalently }
c_d^c(\mu_d) < c_d^c(\delta_1).
$$
\end{corollary}
In fact, Theorem \ref{t:general} follows from the particular case of radii taking only two different values, for which we have more precise results.

\subsection*{Critical intensity  in high dimension : the case of radii taking two values.}

To state the result, we need some further notations.
Fix $\rho>1$. 
Fix $k \ge 1$.
Set $r_1=r_{k+1}=1+\rho$, and for $i \in \{2, \dots, k\}$, $r_i=2$.
For $(a_i)_{2 \le i \le k+1} \in [0,1)^k$, we build an increasing sequence of distances $(d_i)_{1 \le i \le k+1}$ by setting $d_1=1+\rho$ and, for every $i\in \{2, \dots,k+1\}$: 
$$
d_{i}^2=d_{i-1}^2+2r_ia_id_{i-1}+r_i^2.
$$
Note that the sequence $(d_i)_{1 \le i \le k+1}$ depends on $\rho$, $k$, and the $a_i$'s. \\
We set $\L(a_2,\dots,a_{k+1})=d_{k+1}$.
Now set, for every $k \ge 1$, 
\begin{equation}
\label{d:kappack}
\kappa^c_{\rho}(k) = \inf_{0 \le a_2 ,\dots, a_{k+1} < 1} 
\max\left(
\left(\frac{4\rho}{(1+\rho)^2\sqrt{\prod_{2 \le i \le k+1}(1-a_i^2)}}\right)^{\frac1{k+1}},\frac{2\rho}{\L(a_2,\dots,a_{k+1})}
\right).
\end{equation}
Finally, let:
\begin{equation}
\label{d:kappac}
\kappa^c_{\rho} = \inf_{k \ge 1} \kappa^c_{\rho}(k).
\end{equation}
We give some intuition on $\kappa^c_\rho$ in Section \ref{s:ideas}. 
On Figure \ref{f:leskck}, we plot $\kappa^c_\rho(i)$, for $i \in\{1,2,3\}$. 
The data come from the formulas in Lemma \ref{l:kc2} for $i=1$ and from numerical estimations for $i \ge 2$.
\begin{figure}[h!]
\includegraphics[scale=0.7]{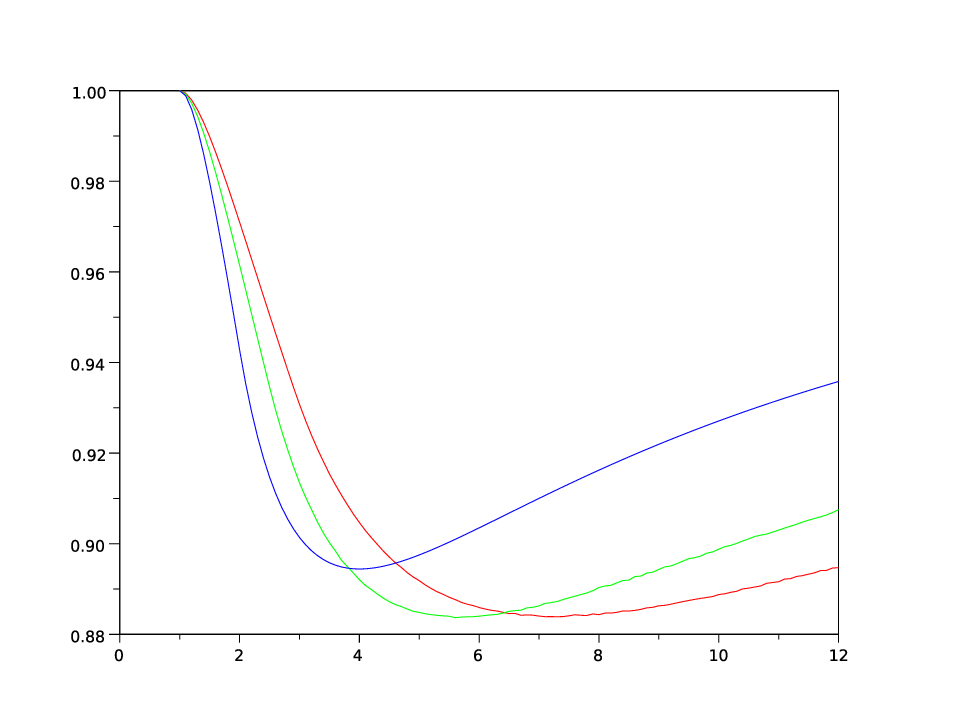}
\caption{$\kappa^c_\rho(i)$, $i \in\{1,2,3\}$, from left to right.}
\label{f:leskck}
\end{figure}

 This $\kappa_{\rho}^c$ gives the asymptotic behaviour of $\widetilde{\lambda}^c_d(\mu_d)$  when $\mu$ charges two distinct points :

\begin{theorem} \label{t:2}
Let $b>a>0$, $\alpha>0$ and $\beta>0$.
Set $\mu = \alpha \delta_a+\beta \delta_b$ and $\rho=b/a>1$.
Then
\begin{equation}\label{e:2}
\lim_{d \to +\infty} \frac1d \ln\left(\widetilde{\lambda}^c_d(\mu_d)\right) = \ln\left(\kappa_{\rho}^c\right)<0.
\end{equation}
\end{theorem}	
If one does not normalize the distribution one has 
\footnote{The upper bound can be proven using $\lambda_d^c(\alpha \delta_a+\beta \delta_b) \le \lambda_d^c(\beta \delta_b)$.
The lower bound can be proven using the easy part of the comparison with a two-type Galton-Watson process.}
$\widetilde{\lambda}_d^c(\alpha \delta_a+\beta \delta_b) \to 1$ and thus 
$\widetilde{\lambda}_d^c(\alpha \delta_a+\beta \delta_b) \sim \widetilde{\lambda}_d^c(\delta_1)$.
This behaviour is due to the fact that, without normalization, the influence of the small balls vanishes in high dimension.

\medskip

In the next lemma we collect some properties of the $\kappa_\rho^c(k)$'s and $\kappa_\rho^c$.
The only result needed for the proof of our main results is $0<\kappa^c_\rho<1$.

\begin{lemma} \label{l:kc2}
Let $\rho>1$. 

$\bullet$  $0<\kappa^c_\rho(1)<1$. More precisely :
$$\text{If } 1<\rho \le 2 \text{ then } \kappa^c_\rho(1) = \frac{2\sqrt{\rho}}{1+\rho}, \quad \text{ while if }
\rho \ge 2 \text{ then } \kappa^c_\rho(1) = \frac{\sqrt{4+\rho^2}}{1+\rho}.$$

$\bullet$ $0<\kappa^c_\rho<1$.

$\bullet$ There exists $\rho_0 > 2$ such that if $\rho \le \rho_0$, then $\kappa^c_\rho=\kappa^c_\rho(1)$. This implies
$$\text{If } 1<\rho \le 2 \text{ then } \kappa^c_\rho=\frac{2\sqrt{\rho}}{1+\rho} , \quad \text{ while if }
2 \le \rho \le \rho_0 \text{ then } \kappa^c_\rho = \frac{\sqrt{4+\rho^2}}{1+\rho}.$$

$\bullet$ As $\rho$ goes to $+\infty$,  $\displaystyle \kappa^c_\rho(k)=1-\frac{k}\rho+o(1/\rho)$. 
Thus one can not restrict the infimum in~\eqref{d:kappac} to a finite number of $k$.
\end{lemma}

Remember that in the case of a deterministic radius (assumptions of Theorem \ref{t:penrose}), 
the first order of the asymptotic behavior of the critical intensity in high dimension is given by 
the threshold of the associated Galton-Watson process.

In the case of two distinct radii (assumptions of Theorem \ref{t:2}), it is thus natural to compare $\kappa_\rho^c$ with the critical parameter of  the associated Galton-Watson process, which is now two-type, one for each radius.

Consider for example $\mu=\lambda(\delta_1+\delta_\rho)$; then $\mu_d=\lambda \delta_1+\frac{\lambda}{\rho^d}\delta_\rho$. Take $\lambda=\frac{\kappa^d}{v_d2^d}$. Now consider the offspring distribution of type $\rho$ of an individual of type~$1$.
We define it to be the number of balls of a Boolean model directed by $\frac{\lambda}{\rho^d} \delta_\rho$ 
that intersects a given ball of radius $1$.
Therefore, this is a Poisson random variable with mean $\frac{\lambda}{\rho^d} v_d(1+\rho)^d$.
The other offspring distribution are defined similarly. 
We moreover assume that the offspring of type $1$ and $\rho$ of a given individual are independent.
The matrix of means of offspring distributions is thus given by:
$$
M_d=
\begin{pmatrix}
\lambda v_d(1+1)^d &  \frac{\lambda}{\rho^d} v_d(1+\rho)^d \\
\lambda v_d(1+\rho)^d & \frac{\lambda}{\rho^d} v_d(\rho+\rho)^d 
\end{pmatrix}= \kappa^d
\begin{pmatrix}
1 &  \left(\frac{1+\rho}{2\rho}\right)^d \\
\left(\frac{1+\rho}{2}\right)^d & 1 
\end{pmatrix} .
$$
Let $r_d$ denotes the largest eigenvalue of $M_d$.
The extinction probability of the two-type Galton-Watson process is $1$ if and only if $r_d \le 1$.
We have:
$$
r_d \sim \left(\frac{\kappa(1+\rho)}{2\sqrt{\rho}}\right)^d, \text{ and thus }\kappa = \frac{2\sqrt{\rho}}{1+\rho} \text{ is the critical parameter.}
$$ 
With Theorem \ref{t:2} and Lemma \ref{l:kc2} we thus see that the comparison with the two-type Galton-Watson is asymptotically 
sharp on a logarithmic scale when $1 < \rho \le 2$, but is not valid for $\rho>2$. This contrasts with the case of a deterministic radius. 

The proof of Theorem \ref{t:penrose} (the constant radii case studied by Penrose) relies on the comparaison with a one-type Galton process which does not depend on $d$.
In the two-values radii case, when $\kappa$ is above or even slighlty below its critical parameter for the Galton-Watson process,
the mean number of children of type $1$ of an individual of type $\rho$ tends to infinity. 
This partly explains why geometrical dependencies can not be handled in the same way as in the constant radii case.
This also partly explains why the critical value for percolation in not always given by the critical value for the Galton-Watson process.

\medskip
The proofs of Theorem \ref{t:2} and of Lemma \ref{l:kc2} are given in Section \ref{s:t:2}.
The main ideas of the proofs are given in Section \ref{s:ideas}.
Theorem \ref{t:general} is an easy consequence of Theorem \ref{t:2}.
The proof is given in Section \ref{s:t:general}.

\section{The case when the radii take two values}
\label{s:t:2}

Before proving Theorem \ref{t:2}, we begin with the proof of Lemma \ref{l:kc2}:

\subsection{Proof of Lemma \ref{l:kc2}} $\;$

\noindent
$\bullet$  By definition, 
$\displaystyle 
\kappa^c_\rho(1)=\inf_{0 \le a < 1} \max(\phi_1(a),\phi_2(a))
$,
where $\phi_1,\phi_2:[0,1) \to \R$ are defined by:
$$
\phi_1(a)=\frac{2\sqrt{\rho}}{(1+\rho)(1-a^2)^{1/4}} \quad \text{ and } \quad 
\phi_2(a)=\frac{\rho\sqrt{2}}{(1+\rho)\sqrt{1+a}}.
$$
If $\rho \le 2$ then $\phi_1(0) \ge \phi_2(0)$.
As $\phi_1$ is increasing and $\phi_2$ is decreasing, we get:
$$
\kappa^c_\rho(1)=\inf_{0 \le a <1} \phi_1(a) = \phi_1(0) = \frac{2\sqrt{\rho}}{1+\rho}.
$$
Assume, on the contrary, $\rho \ge 2$.
Set 
$$
a=\frac{\rho^2-4}{\rho^2+4} \in [0,1).
$$
Then $\phi_1(a) = \phi_2(a)$.
As $\phi_1$ is increasing and $\phi_2$ is decreasing, we get:
$$
\kappa_\rho^c(1)=\phi_1(a)=\phi_2(a)=\frac{\sqrt{4+\rho^2}}{1+\rho}.
$$
$\bullet$  Clearly we have, for every $k \ge 1$ :
\begin{eqnarray*}
\kappa^c_{\rho}(k)  
 & \ge  & 
\inf_{0 \le a_2 ,\dots, a_{k+1} < 1}\left(\frac{4\rho}{(1+\rho)^2\sqrt{\prod_{2 \le i \le k+1}(1-a_i^2)})}\right)^{1/(k+1)} 
  =  \left(\frac{4\rho}{(1+\rho)^2}\right)^{1/(k+1)}.
\end{eqnarray*}
Therefore, as $\kappa^c_{\rho} = \inf_{k \ge 1} \kappa^c_{\rho}(k)$,
$$
 0< \left(\frac{4\rho}{(1+\rho)^2}\right)^{1/2} \le \kappa^c_{\rho} \le \kappa^c_{\rho}(1) <1.
$$
$\bullet$ The last inequalities imply that $\kappa^c_{\rho}=\kappa^c_{\rho}(1)$ if $1<\rho \le 2$. In fact, as soon as 
\begin{equation} \label{e:12}
\kappa^c_{\rho}(1) \le  \left(\frac{4\rho}{(1+\rho)^2}\right)^{1/3},
\end{equation}
it is true that $\kappa^c_{\rho}(1) \le \kappa^c_{\rho}(k)$ for all $k \ge 2$ and therefore that $\kappa^c_{\rho}=\kappa^c_{\rho}(1)$. As the inequality in \eqref{e:12} is strict for $\rho=2$, we obtain by continuity the existence of $\rho_0>2$ such that for every $\rho\in(1,\rho_0)$,    $\kappa^c_\rho=\kappa^c_\rho(1)$.

\noindent
$\bullet$ The minoration follows easily from the following observation: by construction, we have $\mathcal D(a_2, \dots, a_{k+1}) \le 2(\rho +k)$. This implies
$$\kappa_\rho^c(k) \ge \inf_{a_2, \dots a_{k+1}}\frac{2\rho}{\mathcal D(a_2, \dots, a_{k+1})}\ge \frac{\rho}{\rho+k}=1-\frac{k}\rho+o(1/\rho).$$
To obtain the majoration, fix $k\ge 1$. Take $\mu>1/2$ and $\epsilon>0$ such that $\mu+(k-1)\epsilon<1$. Take, for $2 \le i \le k$, the specific values $a_i=\cos(\rho^{-\epsilon})=1+o(\rho^{-\epsilon})$ and $a_{k+1}=\cos(\rho^{-\mu})$. 
Hence, 
$$\frac{4\rho}{(1+\rho)^2\sqrt{\prod_{2 \le i \le k+1}(1-a_i^2)}}\sim 4\rho^{-1+(k-1)\epsilon+\mu}.$$
For $i \le k-1$, we have $\displaystyle 0 \le d_{i+1}-d_i\le d_{i+1}-a_id_i\le \frac{d_{i+1}^2-(a_id_i)^2}{d_{i+1}+a_id_i}\le \frac{2}{\rho+1}$. \\
By summation, we get
$d_k=(1+\rho)(1+\frac{2(k-1)}{1+\rho}+o(\rho^{-1}))$. Now, 
$$d_{k+1}^2  =  d_k^2 +2(1+\rho)d_k\cos(\rho^{-\mu})+(1+\rho)^2, \text{ and thus }
\frac{2\rho}{d_{k+1}}  =  1-\frac{k-1}\rho +o(\rho^{-1}).$$
Finally, $\displaystyle \kappa_\rho^c(k) \le \max \left( \left( \frac{4\rho}{(1+\rho)^2\sqrt{\prod_{2 \le i \le k+1}(1-a_i^2)}}\right)^\frac{1}{k+1}, \frac{2\rho}{d_{k+1}}\right) \le 1-\frac{k-1}\rho +o(\rho^{-1})$.\\
This ends the proof.
\finpreuve


\subsection{Notations and ideas of the proof of Theorem \ref{t:2}}
\label{s:ideas}

In the whole proof, we fix $\rho>1$ and $\kappa>0$.

Once the dimension $d \ge 1$ is given, we consider two independent stationary Poisson point processes on $\R^d$: $\chi_1$ and $\chi_\rho$, with respective intensities 
$$
\lambda_1=\frac{\kappa^d}{v_d 2^d} \quad \text{ and } \quad 
\lambda_{\rho}=\frac{\kappa^d}{v_d 2^d \rho^d}.
$$
With $\chi_1$ and $\chi_{\rho}$, we respectively associate the two Boolean models
$$
\Sigma_1 = \bigcup_{x \in \chi_1} B(x,1) \quad \text{ and } \quad 
\Sigma_{\rho} = \bigcup_{x \in \chi_{\rho}} B(x,\rho).
$$
Note that $\Sigma_{\rho}$ is an independent copy of $\rho \Sigma_1$.
Note also that the expected number of balls of $\Sigma_1$ that touches a given ball of radius $1$ is $\kappa^d$.
Thus the expected number of balls of~$\Sigma_{\rho}$ that touches a given ball of radius $\rho$ is also $\kappa^d$.

We focus on the percolation properties of the following two-type Boolean model
$$
\Sigma = \Sigma_1 \cup \Sigma_{\rho}.
$$
We begin by studying the existence of infinite $k$-alternating paths.
For $k \ge 1$, an infinite $k$-alternating path is an infinite path made of balls such that the radius of the first ball is $\rho$,
the radius of the next $k$ balls is $1$, the radius of the next ball is $\rho$ and so on.
For a fixed $k \ge 1$, we wonder whether infinite $k$-alternating paths exist and seek the critical threshold $\kappa^c_\rho(k)$ for their existence.
A natural first step is to study the following quantities:
\begin{eqnarray}
N_0 & = &  \#\{x_1 \in \chi_\rho : \|x_1\| <2\rho\}, \label{e:Nk}\\
\text{and for } k \ge 1, \quad N_k & = & \#\left\{
\begin{array}{l}
 x_{k+1} \in \chi_\rho: \exists (x_i)_{1 \le i \le k} \in \chi_1 \text{ distinct such that } \\
\|x_1\| < 1+\rho, \; \forall i \in \{1, \dots, k-1\} \, \|x_{i+1}-x_i\| < 2, \\ \|x_{k+1}-x_k\| < 1+\rho
\end{array}
\right\}. \nonumber
\end{eqnarray}
Fix $k \ge 1$. Remember that $\kappa^c_\rho(k)$ is defined in \eqref{d:kappack}.

\medskip
\subsubsection*{A lower bound for $\kappa^c_\rho(k)$} In Subsection \ref {s:subcritical}, we obtain lower bounds for $\kappa^c_\rho(k)$ by looking for upper bounds for $E(N_k)$. On one side, a natural genealogy is associated to the definition of $N_k$ (see also the comments below Theorem \ref{t:penrose} and below Lemma \ref{l:kc2}).
We start with an ancestor $x_0$ located at the origin.
We then seek his children in $\chi_1 \cap B(x_0,1+\rho)$: they constitute the first generation. If $x_1$ is one of those children, we then seek the children of $x_1$ in $\chi_1 \cap B(x_1,2)$ to build the second generation and so on.
On the other side, the process lives in $\R^d$ and the geometry induces dependences: 
if $x_1$ and $x'_1$ are two individuals of the first generation, their children are a priori dependent.
If we forget geometry and only consider genealogy, we get the following upper bound:
$$
E(N_k) \le \lambda_1|B(\cdot,1+\rho)| \left(\prod_{i=2}^k \lambda_1|B(\cdot,2)|\right) \lambda_{\rho}|B(\cdot,1+\rho)|.
$$
But the points of the last generation are in $B(0,2\rho+2k)$.
So if we forget genealogy and only consider geometry we get the following upper bound:
$$ 
E(N_k) \le \lambda_{\rho} |B(0,2\rho+2k)|.
$$
Expliciting the two previous bounds and combining them together, we get:
$$
E(N_k) \le \min \left(\frac{\kappa^{k+1}(1+\rho)^2}{4\rho}, \frac{\kappa (\rho+k)}{\rho}\right)^d.
$$
In this upper bound, the first argument of the  minimum is due to genealogy while the second one is due to geometry.
To get the geometrical term, we considered the worst case: 
the one in which, at each generation $i$, $x_i$ is as far from the origin as possible.
This gives a very poor bound.
To get a better bound, we proceed as follows.
Fix $a_2,\dots,a_{k+1} \in [0,1)$.
As before, we set $r_1=r_{k+1}=1+\rho$, and for $i \in \{2, \dots, k\}$, $r_i=2$ and we build the increasing sequence of distances $(d_i)_{1 \le i \le k+1}$ by setting $d_1=1+\rho$ and, for every $i\in \{2, \dots,k+1\}$: 
$$
d_{i}^2=d_{i-1}^2+2r_ia_id_{i-1}+r_i^2.
$$
See Figure \ref{f:lesdis} for a better understanding of these distances $d_i$.
\begin{figure}[h!]
\scalebox{1.3}{\input{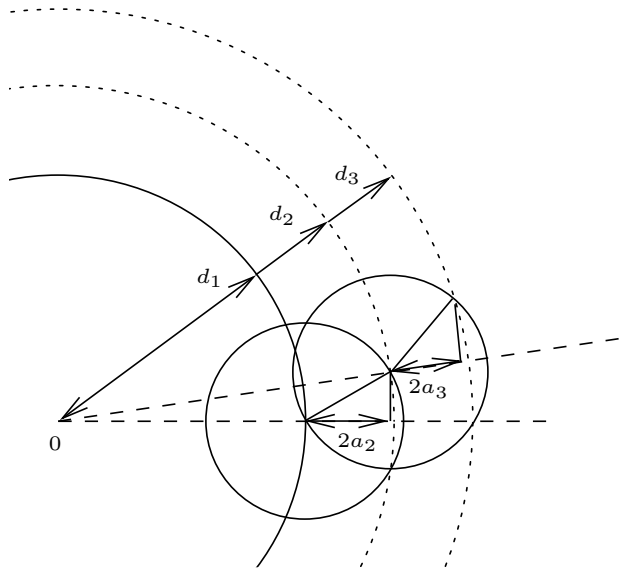}}
\caption{Definition of the distances $d_i$. Circles in plain line are of radius $1+\rho$ and $2$.} 
\label{f:lesdis}
\end{figure}

Denote by $\widetilde{N}_k(a_2,\dots,a_{k+1})$ the number of points $x_{k+1} \in \chi_{\rho}$ 
for which there exists a path $x_1, \dots, x_k$ fulfilling the same requirement as for $N_k$ and such that
$\|x_i\| \approx d_i$ for all~$i$.
Proceeding as before, we obtain the following upper bound:
$$
E(\widetilde{N}_k(a_2,\dots,a_{k+1})) \lesssim 
\min\left(
\kappa^{k+1}\frac{(1+\rho)^2}{4\rho} \sqrt{\prod_{2 \le i \le k+1}(1-a_i)^2},\frac{\kappa \L(a_2,\dots,a_{k+1})}{2\rho}
\right) ^d.
$$
Here again, the first argument of the  minimum is due to genealogy while the second one is due to geometry
\footnote{
There is essentially no geometrical constraint in generations $1$ to $k$.
Very roughly, this is due to the fact that, when $i$ increases from $2$ to $k$ : 
there is more and more space (the $d_i$ are increasing) ;
the intensity of the relevant Poisson point process is the same ; 
the expected number of individuals in the $i^{th}$ generation of the Galton-Watson process decreases.}.  
Optimizing then on the $a_i$'s, we get: 
$$
E(N_k) \lesssim 
\sup_{a_2,\dots,a_{k+1}} \min\left(
\kappa^{k+1}\frac{(1+\rho)^2}{4\rho} \sqrt{\prod_{2 \le i \le k+1}(1-a_i)^2},\frac{\kappa \L(a_2,\dots,a_{k+1})}{2\rho}
\right) ^d.
$$
A precise statement is given in Lemma \ref{l:majorationk}.
The precise value of the threshold $\kappa^c_\rho(k)$ given in \eqref{d:kappack} is then the value such that the above upper bound converges to $0$ when $\kappa<\kappa^c_\rho(k)$. This heuristic will be precised in Subsection \ref {s:subcritical}: we will prove there that when $\kappa < \kappa^c_\rho(k)$, $E(N_k)$ converges to $0$ as $d$ tends to infinity, and this will imply that  there exists no infinite $k$-alternating path.

\medskip
\subsubsection*{An upper bound for $\kappa^c_\rho(k)$}
If, on the contrary, $\kappa > \kappa^c_\rho(k)$ then we will prove that $E(N_k)$ does not converge to $0$.
Actually, to prove that when $\kappa > \kappa^c_\rho(k)$ there exist infinite $k$-alternating path, we will show, in Subsection \ref{s:supercritical}, the following stronger property :
with a probability that converges to $1$ as $d$ tends to infinity, we can find a path which fulfills the
requirements of the definition of $N_k$ -- or more precisely of $\widetilde{N}_k(a_2,\dots,a_{k+1})$ for some $a_2,\dots,a_{k+1}$ nearly optimal -- and which fulfills some extra conditions on the positions of the balls.
This is Proposition \ref{p:onestep} and this is the main technical part of this paper.
Those extra conditions provide independence properties between the existence of different paths of the same kind.
We can then show the existence of many such paths and concatenate some of them to build an infinite $k$-alternating path.
Technically, the last step is achieved by comparing our model with a supercritical oriented percolation process on $\Z^2$.
In this comparison, an open bond in the oriented percolation process corresponds to one of the above paths in our model.
This comparison with oriented percolation was already the last step in the paper of Penrose \cite{Penrose-high-dimensions}.

\medskip
\subsubsection*{From infinite $k$-alternating paths to infinite paths} Recall $\kappa^c_\rho = \inf_{k \ge 1} \kappa^c_\rho(k)$.
With the previous results, 
it is rather easy to show that there is no percolation for $d$ large enough as soon as $\kappa < \kappa^c_\rho$.
When $\kappa>\kappa^c_\rho$ then $\kappa>\kappa^c_\rho(k)$ for a $k \ge 1$.
Therefore there is $k$-alternating percolation and therefore there is percolation.

\subsection{Subcritical phase}
\label{s:subcritical}

Let $\rho>1$ be fixed. We consider, in $\R^d$, the two-type Boolean model $\Sigma$ introduced in Subsection \ref{s:ideas}, with radii $1$ and $\rho$ and respective intensities 
$$
\lambda_1=\frac{\kappa^d}{v_d 2^d} \text{ and } 
\lambda_{\rho}=\frac{\kappa^d}{v_d 2^d \rho^d}
$$
depending on some $\kappa \in (0,1)$.
The aim of this subsection is to prove the following proposition:
\begin{prop} \label{p:subcritical} 
Let $\rho>1$ be fixed. If $\kappa<\kappa^c_\rho$, then, as soon as the dimension $d$ is large enough, percolation does not occur in the two-type Boolean model $\Sigma$.
\end{prop}

In the following of this subsection, we fix $\rho>1$ and  $0<\kappa<\kappa^c_\rho$.

We start with an elementary upper bound, in which we do not take into account the geometrical constraints. We recall that the $N_k$ have been introduced in \eqref{e:Nk}.

\begin{lemma} \label{l:majorationgw}
$E(N_0) = \kappa^d$ and, for $k \ge 1$, 
$\displaystyle 
E(N_k) \le \left(\frac{\kappa^{k+1}(1+\rho)^2}{4\rho}\right)^d
$.
\end{lemma}

\proof
The result for $N_0$ follows directly from the equality 
$
E(N_0) = \lambda_{\rho}|B(0,2\rho)|
$.

Take now $k \ge 1$.
We have :
\begin{equation} \label{e:nkgwinter}
E(N_k) \le \lambda_1|B(\cdot,1+\rho)| \left(\prod_{i=2}^k \lambda_1|B(\cdot,2)|\right) \lambda_{\rho}|B(\cdot,1+\rho)|
\end{equation}
where $B(\cdot,r)$ stands for a ball with radius $r$ and center unspecified.
This can for instance be seen as follows :
\begin{eqnarray*}
E(N_k) 
 & \le & E\left(\sum_{x_1,\dots,x_k \in \chi_1 \text{ distinct}, \;x_{k+1} \in \chi_{\rho}} 1_{x_1 \in B(0,1+\rho)} \dots 1_{x_{k+1} \in B(x_k,1+\rho)}\right) \\
 & = & \lambda_1^k \lambda_{\rho} \int_{\R^{d(k+1)}} dx_1 \dots dx_{k+1} 1_{x_1 \in B(0,1+\rho)} \dots 1_{x_{k+1} \in B(x_k,1+\rho)},
\end{eqnarray*}
which gives \eqref{e:nkgwinter}. The lemma follows.
\finpreuve

\bigskip
To give a more accurate upper bound for the $N_k$'s, we are going to cut the balls into slices and to estimate which slices give the main contribution. 
For $x \in \R^d \setminus \{0\}$, $0 \le a < b \le 1$  and  $r >0$, we now define :
\begin{eqnarray*}
\text{If } a>0 : && B(x,r,a,b) = \left\{ y \in \R^d : \|y-x\| \le r \hbox{ and } ar < \left<y-x, \frac{x}{\|x\|}\right> \le br \right\}, \\
\text{If } a=0 : && B(x,r,a,b) =\left\{ y \in \R^d : \|y-x\| \le r \hbox{ and } \left<y-x,\frac{x}{\|x\|}\right> \le br \right\}.
\end{eqnarray*}
The next lemma gives asymptotics for the volume of these sets:

\begin{lemma} \label{l:gamma} 
For $x \in \R^d \setminus \{0\}$, $0 \le a < b \le 1$  and  $r >0$,
$$
\lim_{d \to + \infty} \frac1d \ln\left(\frac{|B(x,r,a,b)|}{v_d}\right) = \ln\left(r\sqrt{1-a^2}\right).
$$
\end{lemma}
Actually we will only use 
:
$$
\limsup_{d \to + \infty} \frac1d \ln\left(\frac{|B(x,r,a,b)|}{v_d}\right) \le \ln\left(r\sqrt{1-a^2}\right).
$$
\proofof{Lemma \ref{l:gamma}}	
Note that it is sufficient to prove the lemma for $x=e_1$, first vector of the canonical basis, and $r=1$.

First, if  $a=0$, the result follows directly from the inequality $v_d/2 \le |B(e_1,1,0,b)| \le v_d$.

Assume next that $a>0$. On the one hand, $B(e_1,1,a,b)$ is included in the cylinder
$$
\{x=(x_1,\dots,x_d) : x_1 \in [a,1] \hbox{ and } \|(0,x_2,\dots,x_d)\| \le \sqrt{1-a^2}\},
$$
which implies
\begin{equation}\label{e:gammapreuve1}
|B(e_1,1,a,b)| \le v_{d-1} \sqrt{1-a^2}^{d-1} (1-a).
\end{equation}
On the other end, by convexity, $B(e_1,1,a,b)$ contains the following difference between two homothetical cones:
$$
\left\{x=(x_1,\dots,x_d) : x_1 \in [a,b] \hbox{ and } \|(0,x_2,\dots,x_d)\| \le \sqrt{1-a^2} \frac{1-x_1}{1-a}\right\},
$$
which implies
\begin{equation}\label{e:gammapreuve2}
\frac{v_{d-1} \sqrt{1-a^2}^{d-1}}d \big((1-a)-(1-b)\big) \le |B(e_1,1,a,b)|.
\end{equation}
The lemma follows from \eqref{e:gammapreuve1} and \eqref{e:gammapreuve2}. \finpreuve

\bigskip
We can now improve the control given in Lemma \ref{l:majorationgw}:
\begin{lemma} \label{l:majorationk}
For every $k \ge 1$, 
\begin{eqnarray*}
& & \limsup_{d \to \infty}\frac1d \ln(E(N_k))  \\
 & \le & \ln\left(
\sup_{0 \le a_2 ,\dots, a_{k+1} < 1 } \min\left(
\kappa^{k+1}\frac{(1+\rho)^2}{4\rho} \sqrt{\prod_{2 \le i \le k+1}(1-a_i^2)},\frac{\kappa \L(a_2,\dots,a_{k+1})}{2\rho}
\right) \right)  
\end{eqnarray*}
\end{lemma}
\proof
$\bullet$
Fix $N \ge 1$.
Note that the ball $B(x,r)$ is the disjoint union of the slices  $B(x,r,n/N,(n+1)/N)$ for $n \in \{0,\dots,N-1\}$.
For any $n_2,\dots,n_{k+1} \in \{0,\dots,N-1\}$, we set 
$$
a_i=\frac{n_i}{N} \text{ and } a_i^+=\frac{n_i+1}{N}.
$$
We focus on the contribution of a specific product of slices :
$$
N_k(n_2,\dots,n_{k+1})=\#\left\{
\begin{array}{l}
 x_{k+1} \in \chi_\rho: \exists (x_i)_{1 \le i \le k} \in \chi_1 \text{ distinct with} \\
\|x_1\| < 1+\rho, \; \forall i \in \{1, \dots, k-1\} \, x_{i+1} \in B(x_i,2,a_{i+1},a_{i+1}^+), \\ 
x_{k+1} \in  B(x_k,1+\rho,a_{k+1},a_{k+1}^+)
\end{array}
\right\}.
$$
Then we have :
\begin{equation} \label{e:nnk}
N_k \le \sum N_k(n_2,\dots,n_{k+1}),
\end{equation}
where the sum is over $(n_2,\dots,n_{k+1}) \in \{0,\dots,N-1\}^k$.

$\bullet$ As we can check that the points contributing to $N_k(n_2,\dots,n_{k+1})$ are in $B(0,\L(a_2^+,\dots,a_{k+1}^+))$, we get :
$$
E(N_k(n_2,\dots,n_{k+1})) \le \lambda_{\rho}v_d \L(a_2^+,\dots,a_{k+1}^+)^d,
$$
this leads to :
\begin{equation} \label{e:nkgeometrie}
\limsup_{ d \to +\infty} \frac1d \ln\left( E(N_k(n_2,\dots,n_{k+1}))\right) \le 
\ln\left(\frac{\kappa \L(a_2^+,\dots,a_{k+1}^+)}{2\rho}\right).
\end{equation}

$\bullet$ Besides, proceeding as in the proof of Lemma \ref{l:majorationgw}, we obtain :
$$
E(N_k(n_2,\dots,n_{k+1})) \le \lambda_1|B(0,1+\rho)| \left(\prod_{i=2}^k \lambda_1|B(\cdot,2,a_i,a_i^+)|\right) \lambda_{\rho}|B(\cdot,1+\rho,a_{k+1},a_{k+1}^+)|.
$$
With Lemma \ref{l:gamma}, we deduce :
\begin{equation} \label{e:nkgw}
\limsup_{d\to\infty} \frac1d \ln E( N_k(n_2,\dots,n_{k+1})) 
\le \ln\left(\kappa^{k+1}\frac{(1+\rho)^2}{4\rho} \sqrt{\prod_{2 \le i \le k+1}(1-a_i^2)} \right). 
\end{equation}

$\bullet$ 
From \eqref{e:nnk}, \eqref{e:nkgeometrie} and \eqref{e:nkgw} we finally get :
\begin{eqnarray*}
&& \limsup_{d \to +\infty}\frac{\ln(E(N_k))}{d}  \\
& \le & 
\ln\left(
\max_{a_2 ,\dots, a_{k+1}\in \{0,\dots,\frac{N-1}N\}}\min\left(
\kappa^{k+1}\frac{(1+\rho)^2}{4\rho} \sqrt{\prod_{2 \le i \le k+1}(1-a_i^2)},\frac{\kappa \L(a_2^+,\dots,a_{k+1}^+)}{2\rho}
\right)
\right) .
\end{eqnarray*}
As $\L$ is uniformly continuous on $[0,1]^k$, we end the proof by taking the limit when $N$ goes to $+\infty$. \finpreuve

\bigskip
The next step consists in taking into account all $k\ge 0$ simultaneously; we thus introduce

\begin{equation} \label{e:N}
N=\#\left(\left\{
\begin{array}{l}
 y \in \chi_\rho: \exists k\ge 1, \exists (x_i)_{1 \le i \le k} \in \chi_1 \text{ distinct with} \\
\|x_1\| < 1+\rho, \; \forall i \in \{1, \dots, k-1\} \, \|x_{i+1}-x_i\| < 2, \\ \|y-x_k\| < 1+\rho
\end{array}
\right\} \cup \{y \in \chi_\rho : \|y\| < 2\rho\}\right).
\end{equation}

\begin{lemma} \label{l:majoration}
If $\kappa < \kappa^c_\rho$, then 
$\displaystyle
\limsup_{d \to +\infty} \frac1d \ln(E(N)) < 0
$.
\end{lemma}

\proof
We have :
\begin{equation}\label{e:majEN}
E(N) \le \sum_{k \ge 0} E(N_k).
\end{equation}
As $\kappa < \kappa^c_\rho$, Lemma \ref{l:majorationk} ensures that for every $k \ge 1$ :
\begin{equation} \label{e:kfixe}
\limsup_{d\to+\infty} \frac1d \ln(E(N_k)) < 0.
\end{equation}
Moreover, the assumption $\kappa < \kappa^c_\rho$ also implies, thanks to Lemma \ref{l:kc2}, that $\kappa < 1$.
We can then choose $k_0$ large enough to have :
$$
\frac{\kappa^{k_0+1}(1+\rho)^2}{4\rho} \le \exp(-1).
$$
With Lemma \ref{l:majorationgw}, we thus get :
$$
E(N_0) + \sum_{k \ge k_0} E(N_k) \le \kappa^d + \exp(-d)\sum_{k \ge 0} \kappa^{kd} = \kappa^d + \exp(-d)\frac{1}{1-\kappa^d}.
$$
With \eqref{e:majEN} and \eqref{e:kfixe}, this ends the proof. \finpreuve

\bigskip
The next lemma is elementary

\begin{lemma} \label{l:chi1nonperco} 
Assume $\kappa<1$. \\
Then the connected components of $\displaystyle \bigcup_{x \in \chi_1} B(x,1)$ are bounded with probability $1$.
\end{lemma}

\proof
For any integer $k \ge 0$, denote by $M_k$ the number of balls with radius $1$ linked to $B(0,1)$ by a chain of $k$ distinct balls with radius $1$.
Proceeding as in the proof of Lemma \ref{l:majorationgw}, we get :
$$
E(M_k) \le \kappa^{d(k+1)}.
$$
Now denote by $M$ the number of balls with radius  $1$ linked to  $B(0,1)$ by a chain of  (perhaps no) balls with radius $1$.
Then :
$$
E(M) \le E\left(\sum_{k \ge 0} M_k\right) =\frac{\kappa^d}{1-\kappa^d} < +\infty.
$$
Therefore, $M$ is finite with probability $1$. So the connected components that touch $B(0,1)$ are bounded with probability $1$. So with probability $1$, every connected component is bounded.  
\finpreuve

\proofof{Proposition \ref{p:subcritical}}
Remember that we proved in Lemma \ref{l:kc2} that $\kappa^c_\rho<1$.
Take $\kappa$ such that $0 < \kappa < \kappa^c_\rho<1$. 

Let $\xi_1$ be the set of random balls with radius $\rho$ that can be connected to $B(0,\rho)$ through a chain of random balls with radius $1$
(we consider the condition as fulfilled if the ball touches $B(0,\rho)$ directly).
Let $\xi_2$ be the set of random balls with radius $\rho$ that are not in~$\xi_1$, 
but that can be connected to $B(0,\rho)$ through a path of random balls in which there is only one ball with radius $\rho$.
We define similarly $\xi_3, \xi_4$ and so on and denote by $\xi$ the disjoint union of all these sets.

We have $\# \xi_1 = N$. (Remember that $N$ has been defined in \eqref{e:N}.)
By Lemma \ref{l:majoration}, we have :
$$
\limsup_{d \to +\infty} \frac1d \ln(E( \xi_1)) < 0.
$$
Take some $\mu>0$ and assume from now on that $d$ is large enough to have
$$
\frac1d \ln(E(\# \xi_1)) \le -\mu.
$$ For every $k \ge 1$, we have then:
$
E(\# \xi_k) \le \left(E(\# \xi_1)\right)^k \le \exp(-dk \mu) 
$.

As $\xi = \cup_{k \ge 1} \xi_k$, we deduce from the previous inequalities that $\xi$ is finite with probability $1$. 
So if an unbounded connected component of 
 $\Sigma_1 \cup \Sigma_{\rho}$ touches $B(0,\rho)$ then there is an unbounded component in $\Sigma_1$. 
As $\kappa<1$, Lemma \ref{l:chi1nonperco} rules out the possibility of an unbounded connected component in $\Sigma_1$.
So with probability $1$, the connected components of $\Sigma_1 \cup \Sigma_{\rho}$ that touch $B(0,\rho)$ are bounded, which ends the proof. \finpreuve

\subsection{Supercritical phase}
\label{s:supercritical}

We fix here $\rho>1$. We consider once again the two-type Boolean model $\Sigma$ introduced in Subsection \ref{s:ideas} and  we fix an integer $k \ge 1$.

For every $n \ge 0$, we set $R_n=\rho$ if $k+1$ divides $n$ and $R_n=1$ otherwise.
We say that percolation by $k$-alternation occurs if there exists an infinite sequence  of distinct points $(x_n)_{n \in \N}$ in $\R^d$  such that, for every $n\ge 0$: 
\begin{itemize}
\item $x_n \in \chi_{R_n}$.
\item $B(x_n,R_n) \cap B(x_{n+1},R_{n+1}) \neq  \emptyset$.
\end{itemize}
In other words, percolation by $k$-alternation occurs if there exists an infinite path along which $k$ balls of radius $1$ alternate with one ball of radius $\rho$, \emph{ie} if there exists an infinite $k$-alternating path.
The aim of this subsection is to prove the following proposition:

\begin{prop} \label{p:supercritical}
Let $\rho>1$ and $k \ge 1$ be fixed. Assume that $\kappa \in (\kappa^c_{\rho}(k),1)$.
If the dimension $d$ is large enough, then percolation by $k$-alternation occurs with probability one. 
\end{prop}

As announced in Subsection \ref{s:ideas}, percolation by $k$-alternation of the two-type Boolean model in the supercritical case will be proved by embedding in the model a supercritical $2$-dimensional oriented percolation process.

We thus specify the two first coordinates, and introduce the following notations.
When $d \ge 3$, for any $x \in \R^d$, we write 
$$
x=(x',x'') \in \R^2 \times \R^{d-2}.
$$ 
We write $B'(c,r)$ for the open Euclidean balls of $\R^2$ with center $c\in \R^2$ and radius $r>0$.
In the same way we denote by $B''(c,r)$ the open Euclidean balls of $\R^{d-2}$ with center $c \in \R^{d-2}$ and radius $r>0$.

\subsubsection{One step in the $2$-dimensional oriented percolation model}
\label{s:onestep}

The point here is to define the event that will govern the opening of the edges in the $2$-dimensional oriented percolation process : it is naturally linked to the existence of a finite path composed of $k$ balls of radius $1$ and a ball of radius $\rho$, whose positions of centers are specified.
 
We define, for a given dimension $d$, the two following subsets of $\R^d$ :
\begin{eqnarray*}
W & = &  d^{-1/2} \left((-1,1) \times (-1,0) \times \R^{d-2}\right), \\
W^+ & = &  d^{-1/2} \left((0,1) \times (0,1) \times \R^{d-2}\right). 
\end{eqnarray*}
For  $x_0 \in W$ we set :
\begin{equation} \label{e:bonevnt}
{\mathcal G}^+(x_0) = \left\{
\begin{array}{c}
\hbox{There exist distinct } x_1,\dots,x_k \in \chi_1 \cap  W^+ \hbox{ and } x_{k+1} \in \chi_{\rho} \cap W^+ \\
\hbox{such that } x_0, x_1, \dots, x_{k+1} \hbox{ is a path}
\end{array}
\right\}.
\end{equation}
Our goal here is to prove that the probability of occurrence of this event is asymptotically large :

\begin{prop} \label{p:onestep} 
Let $\rho>1$ and $k \ge 1$ be fixed. Assume that $\kappa \in (\kappa^c_\rho(k),1)$ and choose $p \in (0,1)$. If the dimension $d$ is large enough, then for every $x \in W$,
$$
P({\mathcal G}^+(x))\ge p.
$$
\end{prop}
Note already that by translation invariance, $P({\mathcal G}^+(x))$ does not depend on $x''$, so we can assume without loss of generality that $x''=0$. In the sequel of this subsection, $\rho>1$ and $k \ge 1$ are fixed. 

\bigskip

We first recall the definitions of the $(d_i)_{1 \le i \le k+1}$ and of $
\kappa^c_{\rho}(k)$ we give in the introduction. We set  $r_1=r_{k+1}=1+\rho$ and for $i \in \{2, \dots, k\}$, $r_i=2$.
Then, for a given sequence $(a_i)_{2 \le i \le k+1} \in [0,1)^k$, we build an increasing sequence $(d_i)_{1 \le i \le k+1}$  by setting $d_1=1+\rho$ and, for every $i\in \{2, \dots,k+1\}$ :
$$
d_{i}^2=d_{i-1}^2+2r_ia_id_{i-1}+r_i^2.
$$
Finally, we note $\L(a_2,\dots,a_{k+1})=d_{k+1}$ and we set
$$
\kappa^c_{\rho}(k) = \inf_{0 \le a_2 ,\dots, a_{k+1} < 1} 
\max\left(
\left(\frac{4\rho}{(1+\rho)^2\sqrt{\prod_{2 \le i \le k+1}(1-a_i^2)}}\right)^{1/(k+1)},\frac{2\rho}{\L(a_2,\dots,a_{k+1})}
\right).
$$
The first step consists in choosing a nearly optimal sequence $(a_i)_{2 \le i \le k+1} \in [0,1)^k$ satisfying some extra inequalities :

\begin{lemma}
We can choose $(a_i)_{2 \le i \le k+1} \in [0,1)^k$ such that :
\begin{equation}
1<\kappa^{k+1}\frac{(1+\rho)^2}{4\rho}\sqrt{\prod_{2 \le j \le k+1}(1-a_j^2)}<\kappa \frac{d_{k+1}}{2\rho}. \label{cond}
\end{equation}
\end{lemma}

\begin{proof}
As $\kappa^c_\rho(k)<\kappa$, we can choose $(a_i^0)_{2 \le i \le k+1} \in (0,1)^k$ such that the two following conditions
\begin{eqnarray}
 \kappa & > & \left( \frac{4\rho}{(1+\rho)^2\sqrt{1-a_2^2}\dots\sqrt{1-a_{k+1}^2})}\right)^{\frac1{k+1}}, \label{CGW} \\
 \kappa & > & \frac{2\rho}{d_{k+1}} \label{Cvolume} 
\end{eqnarray}
are fullfilled for $(a_i)_i = (a_i^0)_i$. We fix
$(a_3, \cdots, a_{k+1})=(a_3^0, \cdots, a_{k+1}^0)$.
Note that
\begin{eqnarray*}
f: \; a_2 \mapsto \kappa \frac{d_{k+1}}{2\rho}&& \hbox{is continuous and increasing}, \\
g: \; a_2 \mapsto \kappa^{k+1}\frac{(1+\rho)^2}{4\rho}\sqrt{\prod_{2 \le j \le k+1}(1-a_j^2)}&& \hbox{is continuous and decreasing},
\end{eqnarray*}
and that $\lim_{a_2 \to 1}g(a_2)=0$. Moreover, Conditions \eqref{Cvolume} and \eqref{CGW} ensure that $f(a_2^0)>1$ and $g(a_2^0)>1$. 

Thus if $f(a_2^0) > g(a_2^0)$ the proof is over.
If $f(a_2^0) \le g(a_2^0)$, we can take $a_2 > a_2^0$ such that $1<g(a_2)<f(a_2^0)$ : then $f(a_2) \ge f(a_2^0) > g(a_2) > 1$ and the lemma is proved. \finpreuve
\end{proof}

\bigskip
Note that \eqref{cond} implies \eqref{CGW} and \eqref{Cvolume}.

\bigskip
As explained in Subsection \ref{s:subcritical}, the main contribution to the number $N_k$ of centers $x_{k+1}$ of balls of radius $\rho$ that are linked to a ball of radius $\rho$ centered at the origin by a chain $(x_i)_{1 \le i \le k}$ of $k$ balls of radius $1$ -- see the precise definition \eqref{e:Nk} -- is obtained for $\|x_i\| \sim d_i$, where the $d_i's$ are build from a (nearly) optimal sequence $(a_i)_{2 \le i \le k+1} \in [0,1)^k$. So we fix a nearly optimal family $(a_i)_{2 \le i \le k+1} \in (0,1)^k$ satisfying \eqref{cond}, we build the associated family of distances $(d_i)_{1 \le i \le k+1} \in (0,1)^k$ and we are going to look for a good sequence of centers $(x_i)_{1 \le i \le k+1} \in (0,1)^k$ with $\|x_i\| \sim d_i$. 

We thus introduce the following subsets of $\R^2$ :
\begin{eqnarray*}
D'_0 & = & \big(-d^{-1/2},d^{-1/2}\big) \times \big( -d^{-1/2}, 0 \big), \\
\forall i \in \{1, \dots, k+1\} \quad D'_i & = & \big(0,d^{-1/2}\big)^2,
\end{eqnarray*}
and the followining sets in $\R^{d-2}$
\begin{eqnarray*}
C''_0 & = & \{0\}, \\
\forall i \in \{1, \dots, k+1\} \quad C''_i & = & B''(0, d_i-2d^{-1}) \setminus B''(0, d_i-3d^{-1}).
\end{eqnarray*}
Finally, for $i \in \{0,\dots,k+1\}$, we set
$
C_i   =   D'_i\times C''_i
$. Note that for $d$ large enough, these sets are disjoint. The next lemma controls the asymptotics in the dimension $d$ of the volume of these sets

\begin{lemma} \label{l:c}
For every $i \in \{1,\dots,k+1\}$ :
$$
\lim_{d \to +\infty} \frac1d \ln \frac{|C''_i|}{v_{d-2}}=\lim_{d \to +\infty} \frac1d \ln \frac{|C_i|}{v_{d}}=\ln d_i.
$$ 
\end{lemma}
\proof This can be proven by elementary computations.  \finpreuve

\bigskip
Each $x_i$ will be taken in $C_i$, but we also have to ensure that the $(x_i)_{1 \le i \le k+1}$ form a path. 
Note that for $i \in \{2,\dots,k+1\}$, we have $0<d_{i-1}+a_ir_i<d_i$, which legitimates the following definition. See also Figure \ref{f:lesdis}.
For $i\in \{2, \dots, k+1\}$ and $d$ large enough, we denote by $\theta_i$ the unique real number in $(0, \pi/2)$ such that 
$$
\cos \theta_i =\frac{d_{i-1}+a_ir_i}{d_i}+d^{-1/2}.
$$
We introduce next, for $y \in C_{i-1}$, the following subset of $\R^{d-2}$ :
$$
D''_i(y'') =  \{z'' \in C''_i : \; \left<z'',y''\right>\ge \|y''\|.\|z''\|. \cos \theta_i\}
$$ 
We also set $D''_0=C''_0$ and $D''_1(y'')=C''_1$ for every $y \in C_0$.
Finally, we define for every $i \in \{1,\dots,k+1\}$ and $y \in C_{i-1}$ :
$$
D_i(y) =  D'_i \times D''_i(y'')\subset C_i,
$$
and $D_0=D'_0 \times D''_0$.

\begin{lemma} \label{l:G}
$\bullet$ If the dimension $d$ is large enough, for every $i\in \{1, \dots, k+1\}$ and $y \in C_{i-1}$,
$$
D_{i}(y) \; \subset \; B(y,r_i) \cap C_{i}.
$$
$\bullet$ Let $x_0 \in D_0$. If there exist $X_1, \dots, X_k \in \chi_1$ and $X_{k+1} \in \chi_{\rho}$ such that 
$X_1 \in D_1(x_0), X_2 \in D_2(X_1), \dots, X_{k+1} \in D_{k+1}(X_k)$, then the event ${\mathcal G}_+(x_0)$ occurs.
\end{lemma}

\proof 
$\bullet$ The inclusion $D_i(y) \subset C_i$ is clear for every $i \in \{1,\dots,k+1\}$.
Let $i \in \{2, \dots, k+1\}$, $y \in C_{i-1}$ and $z \in D_{i}(y)$. Then, as soon as $d$ is large enough,
\begin{eqnarray*}
 \|z-y\|^2 
 & = & \|z'-y'\|^2+ \|z''-y''\|^2 \\
 & \le & \frac2d+\|y''\|^2+\|z''\|^2-2<y'',z''> \\
 & \le & \frac2d+(d_{i-1}-2d^{-1})^2+(d_i-2d^{-1})^2-2(d_{i-1}-3d^{-1})(d_i-3d^{-1}) \cos \theta_i \\
 & \le & d_i^2 +d_{i-1}^2-2d_{i-1}(d_{i-1}+a_ir_i)-2d^{-1/2}d_id_{i-1}+O_i(d^{-1}) \\
 & \le & r_i^2-2d^{-1/2}d_id_{i-1}+O_i(d^{-1}) \le r_i^2. 
\end{eqnarray*}
Let now $y \in C_0$ and $z \in D_1(y)$. As $d_1=1+\rho=r_1 > 2$, we obtain, for $d$ large enough : 
$$
\|z-y\|^2 =  \|z'-y'\|^2+ \|z''-y''\|^2  \le \frac8d+(d_1-2d^{-1})^2 \le r_1^2.
$$
$\bullet$ The second point is a simple consequence of the first point, of the fact that the sets $D_i(x_{i-1})$, as the sets $C_i$, are disjoint and of the definition of the event ${\mathcal G}_+$. 
\finpreuve

\bigskip
Note that for $i \in \{1,\dots,k+1\}$, $|D_i(y)|$ and $|D''_i(y'')|$ do not depend on the choice of $y \in C_{i-1}$.
We thus denote by $|D_i|$ and $|D''_i|$ these values. We now give asymptotic estimates for these values :

\begin{lemma} \label{l:d}
For every $i \in \{2,\dots,k+1\}$, 
$$
\lim_{d \to +\infty} \frac1d \ln \frac{|D''_{i}|}{v_{d-2}} = \lim_{d \to +\infty} \frac1d \ln \frac{|D_{i}|}{v_d} =\ln (r_i \sqrt{1-a_i^2}).
$$
\end{lemma}

\proof
We have, by homogeneity and isotropy:
\begin{equation} \label{e:louis}
|D''_i| = \left((d_i-2d^{-1})^{d-2}-(d_i-3d^{-1})^{d-2}\right) |S|
\end{equation}
where
$
S = \{x=(x_1,\dots,x_{d-2}) \in B''(0,1) : x_1 \ge \|x\| \cos(\theta_i)\}.
$\\
But $S$ is included in the cylinder 
$$
\{(x_i)_{1 \le i \le d-2} \in \R^{d-2} : x_1 \in [0,1], \|(x_2,\dots,x_{d-2})\| \le \sin(\theta_i)\}
$$
and $S$ contains the cone 
$$
\{ (x_i)_{1 \le i \le d-2}\in \R^{d-2} : x_1 \in [0,\cos(\theta_i)], \|(x_2,\dots,x_{d-2})\| \le x_1\sin(\theta_i)\cos(\theta_i)^{-1}\}.
$$
Therefore :
\begin{equation} \label{e:louis2}
v_{d-3}\cos(\theta_i) \sin(\theta_i)^{d-3}(d-2)^{-1} \le |S| \le v_{d-3}\sin(\theta_i)^{d-3}.
\end{equation}
From \eqref{e:louis}, \eqref{e:louis2}, and the limits $\cos(\theta_i) \to (d_{i-1}+a_ir_i)d_i^{-1} \neq 0$
and $d_i \sin(\theta_i) \to r_i \sqrt{1-a_i^2}$, we get 
$$ \lim_{d \to +\infty} \frac1d \ln\left(\frac{|D_i''|}{v_{d-2}}\right) = \ln(r_i \sqrt{1-a_i^2}).
$$
The lemma follows. 
Note that a direct calculus with spherical coordinates can also give the announced estimates.\finpreuve

%

\bigskip
Everything is now in place to prove Proposition \ref{p:onestep}.

\proofof{Proposition \ref{p:onestep}}
Choose $p<1$ and  $x \in W$ such that $x''=0$.

$\bullet$ 
We start with a single individual, encoded by its position $\zeta_0=\{x\} \subset C_0$, and we build, generation by generation, its descendance : we set, for $1 \le i \le k$,
$$
\zeta_i = \chi_1 \cap \bigcup_{y \in \zeta_{i-1}} D_{i}(y) \subset C_i,
$$
and for the $(k+1)$-th generation, we finally set
$$
\zeta_{k+1} = \chi_{\rho} \cap \bigcup_{y \in \zeta_k} D_{k+1}(y) \subset C_{k+1}. 
$$ 
By Lemma \ref{l:G}, if $\zeta_{k+1}\neq \varnothing$ then the event ${\mathcal G}^+(x)$ occurs. To bound from below the probability that $\zeta_{k+1} \neq \varnothing$, we now build a simpler process $\xi$, stochastically dominated by~$\zeta$.

$\bullet$ We set $\alpha_i=\lambda_1 |D_i|$ for $i \in \{1,\dots,k\}$ and $\alpha_{k+1}=\lambda_{\rho} |D_{k+1}|$ : thus, $\alpha_i$ is the mean number of children of a point of the $(i-1)$-th generation. Let $X_0=x$ be the position of the first individual. 

Consider a random vector $X=(X_0,X_1,\dots,X_{k+1})$ of points in $\R^d$ defined as follows : $X_0$ is defined by $X_0=x$,
$X_1$ is taken uniformly in $D_1(X_0)$, then $X_2$ is taken uniformly in $D_2(X_1)$, and so on.
We think of $X$ as a pontential single branch of descendance of $x$.
Let then $(X^j)_{j \ge 1}$ be independent copies of $X$.
Let now $N$ be an independent Poisson random variable with parameter $\alpha_1$ : this random variable will be the number of children of $X_0$.
We will use the $N$ first $X^j$, one for each child of $X_0$.

We now take into account the fact that some individuals may have no children. We shall deal with geometric dependencies later. Note that in our new process each individual of any generation $i, \,i \ge 1$, has at most one child. We made that choice in order to handle more easily geometric dependencies.
Let $Y=(Y_i^j)_{2 \le i \le k+1, j \ge 1}$ be an independent family of independent random variables, such that $Y^j_i$ follows the Bernoulli law with parameter $1-\exp(-\alpha_i)$, which is the probability that a Poisson random variable with parameter $\alpha_i$ is different from $0$. 
We set $J_1=\{1,\dots,N\}$ and, for every $i \in \{2, \dots, k+1\}$ :
$$
J_i = \{ 1 \le j \le N : Y_2^j = \dots = Y_i^j = 1\}.
$$
Thus the random set $J_i$ gives the exponents of the branches that are, among the $N$ initial branches, still alive at the $i$-th generation  in a process with no dependecies due to geometry.

Until now, we did not take into account the geometrical constraints between individuals. For every $i \in \{2,\dots,k+1\}$ and every $j \ge 1$, we set
$$
Z_i^j=1 \text{ if }X_i^j \not\in \bigcup_{j' \in J_{i-1} \setminus \{j\}} D_i(X_{i-1}^{j'}) \quad \quad \text{ and } Z_i^j=0 \text{ otherwise}.
$$
We will reject an individual $X_i^j$ and its descendance as soon as $Z_i^j=0$.
Recall that, when building generation $i$ from generation $i-1$, we explore the Poisson point processes in the area $\bigcup_{j \in J_{i-1}} D_i(X_{i-1}^{j})$. 
By construction of the $C_i$, these areas are distinct for different generations. 
Therefore, one can check that, for every $i \in \{2,\dots, k+1\}$, the set
$$
\xi_i=\{X_i^j :  j \in J_i \hbox{ and } Z_2^j=\dots=Z_i^j=1\}
$$
is stochastically dominated  by $\zeta_i$. Thus to prove Proposition \ref{p:onestep}, we now need to bound from below the probability that $\xi_{k+1}$ is not empty.


$\bullet$ Let $T$ be the smallest integer $j$ such that $ Y_2^j=\dots=Y_{k+1}^j=1$ : in other words, $T$ is the smallest exponent of a branch that lives till generation $k+1$. To ensure that $\xi_{k+1}\neq \varnothing$, it is sufficient that $T\le N$ and that $Z_2^T =\dots=Z_{k+1}^T=1$. So :
\begin{eqnarray*}
1-P({\mathcal G}^+(x)) 
 & \le & P(\xi_{k+1}=\emptyset) \\
 & \le & P(\#J_{k+1} = 0) + P \left(\{T \le N\} \cap  \bigcup_{2 \le i \le k+1} \{Z_i^T = 0\}\right) \\
 & \le & P(\#J_{k+1} = 0) + \sum_{2 \le i \le k+1} P(T \le N \hbox{ and } Z_i^T=0).
\end{eqnarray*}
For every $2 \le i \le k+1$, we have by construction :
\begin{eqnarray*}
P(T \le N \hbox{ and } Z_i^T=0) 
 & = & P\left(T \le N, \; \exists j \in J_{i-1} \setminus \{T\}  \hbox{ such that } X_i^T \in D_i(X_{i-1}^j) \right) \\
 & \le & \sum_{j \ge 1} P\left(T \le N \hbox{ and } j \in J_{i-1} \setminus \{T\}  \hbox{ and } X_i^T \in D_i(X_{i-1}^j) \right) \\
 & = & \sum_{j \ge 1} E\left( 1_{T \le N} 1_{j \in J_{i-1} \setminus \{T\}} P\left( X_i^T \in D_i(X_{i-1}^j) \left.\right| Y,N \right)\right) \\
 & = & \sum_{j \ge 1} E\left( 1_{T \le N} 1_{j \in J_{i-1} \setminus \{T\}} \right)P\left( X_i^1 \in D_i(X_{i-1}^2) \right) \\
 & \le & E(\# J_{i-1})P\left( X_i^1 \in D_i(X_{i-1}^2) \right).
\end{eqnarray*}
Besides, as  $(X_i^1)''$ is uniformly distributed on  $C''_i$  and is independent of $(X_{i-1}^2)''$, 
\begin{equation}
P\left(X_i^1 \in D_i(X_{i-1}^2)\right) 
  =   P\left((X_i^1)'' \in D''_i((X_{i-1}^2)'')\right) 
  =  \frac{|D''_i|}{|C''_i|} \label{e:probainterference} .
\end{equation}
This leads to 
\begin{equation}
1-P({\mathcal G}^+(x)) \le P(\#J_{k+1} = 0) + \sum_{i=2}^{k+1} E(\# J_{i-1})\frac{|D''_i|}{|C''_i|}. \label{e:majJ}
\end{equation}

$\bullet$ For $1 \le i \le k+1$, the cardinality of $J_i$ follows a Poisson law with parameter
$$
\eta_{i} = \alpha_1 \prod_{i'=2}^{i} (1-\exp(-\alpha_{i'})).
$$
Remember that  $\alpha_i=\lambda_1 |D_i|$ for $i \in \{1,\dots,k\}$ and $\alpha_{k+1}=\lambda_{\rho} |D_{k+1}|$. By Lemma \ref{l:d}, we have the following limits:
\begin{eqnarray*}
\lim_{d \to +\infty} \frac1d \ln \alpha_1 & = & \ln \frac{\kappa(1+\rho)}{2} >  0, \\
\lim_{d \to +\infty} \frac1d \ln \alpha_i & = & \ln (\kappa\sqrt{1-a_i^2}) < 0\hbox{ for } 2 \le i \le k, \\
\lim_{d \to +\infty} \frac1d \ln  \alpha_{k+1} & = & \ln(\kappa\sqrt{1-a_{k+1}^2}\frac{1+\rho}{2\rho}) < 0.
\end{eqnarray*}
To see the signs of the limits, note that $\kappa<1$, that $\frac{1+\rho}{2\rho}<1$ and that \eqref{cond} implies that
$$\kappa> \kappa^{k+1} > \frac{4\rho}{(1+\rho)^2 \sqrt{1-a_2^2}\dots\sqrt{1-a_{k+1}^2}}>\frac{2}{1+\rho}.$$
Consequently, we first see that 
\begin{eqnarray}
\lim_{d\to +\infty}  \frac1d \ln(\eta_{k+1}) 
 & = & \lim_{d\to +\infty}  \frac1d \ln(\alpha_1\dots\alpha_{k+1}) \nonumber \\ 
 & = & \ln \left(\kappa^{k+1}\frac{(1+\rho)^2}{4\rho}\sqrt{\prod_{2 \le j \le k+1}(1-a_j^2)} \right)> 0 \quad \text{ with \eqref{cond};} \nonumber \\
 \text{therefore,} &&\lim_{d\to +\infty} P(\#J_{k+1}=0) = 0. \label{e:premierterme}
\end{eqnarray}
Similarly, for $2 \le i \le k+1$, we have
\begin{eqnarray*}
\lim_{d\to +\infty} \frac1d \ln(\eta_{i-1}) 
 & = & \ln \left(\kappa^{i-1}\frac{1+\rho}{2}\sqrt{\prod_{2 \le i' \le i-1}(1-a_{i'}^2)}\right).
\end{eqnarray*}
Lemmas \ref{l:c} and \ref{l:d} ensure that :
$$
\lim_{d\to +\infty} \frac1d  \ln \left(\frac{|D''_i|}{|C''_i|}\right) = \ln\left(\frac{r_i\sqrt{1-a_i^2}}{d_i}\right).
$$
Thus, for $2 \le i \le k+1$, we have :
$$
\limsup_{d\to +\infty} \frac1d \ln \left( E(\# J_{i-1})\frac{|D''_i|}{|C''_i|}\right) 
\le  \ln \left( \frac{r_i (1+\rho) \kappa^{i-1}}{2d_i}\sqrt{\prod_{2 \le i' \le i}(1-a_{i'}^2)}\right).
$$
Now,
\begin{eqnarray} 
 && \text{for } 2 \le i \le k, \quad \limsup_{d\to +\infty} \frac1d \ln \left( E(\# J_{i-1})\frac{|D''_i|}{|C''_i|}\right)  \le  \ln \left( \frac{1+\rho}{d_i}\right) < 0, \label{e:deuxiemeterme}\\
&& \limsup_{d\to +\infty} \frac1d \ln \left( E(\# J_k)\frac{|D''_{k+1}|}{|C''_{k+1}|}\right) 
 \le  \ln \left( \frac{(1+\rho)^2 \kappa^k}{2d_{k+1}}\sqrt{\prod_{2 \le i' \le k+1}(1-a_{i'}^2)}\right)<0 \label{e:troisiemeterme}
\end{eqnarray}
with \eqref{cond}. To end the proof, we put estimates \eqref{e:premierterme}, \eqref{e:deuxiemeterme} and \eqref{e:troisiemeterme} in \eqref{e:majJ}.
\finpreuve

\subsubsection{Several steps in the $2$-dimensional oriented percolation model}
\label{s:severalsteps}

We prove here Proposition \ref{p:supercritical} by building the supercritical $2$-dimensional oriented percolation process embedded in the two-type Boolean Model.

\proofof{Proposition \ref{p:supercritical}}
We first define an oriented graph in the following manner: the set of sites is 
$$S=\{(a,n)\in \Z \times \N: \; |a| \le n, \; a+n \hbox{ is even }\};$$
from any point $(a,n) \in S$, we put an oriented edge from $(a,n)$ to $(a+1,n+1)$, and an oriented edge from $(a,n)$ to $(a-1, n+1)$. We denote by 
$\vec{p}_c(2) \in (0,1)$ the critical parameter for Bernoulli percolation on this oriented graph -- see Durrett \cite{Durrett-oriented} 
 for results on oriented percolation in dimension 2.  

For any $(a,n) \in S$, we define the following subsets of $\R^d$ 
\begin{eqnarray*}
W_{a,n} & = &  d^{-1/2} \left(]a-1,a+1[ \times ]n-1,n[ \times \R^{d-2}\right), \\
W^-_{a,n} & = &  d^{-1/2} \left(]a-1,a[ \times ]n,n+1[ \times \R^{d-2}\right), \\
W^+_{a,n} & = &  d^{-1/2} \left(]a,a+1[ \times ]n,n+1[ \times \R^{d-2}\right).
\end{eqnarray*}
Note that the $(W_{a,n})_{(a,n) \in S}$ are disjoint and that $W^+_{a,n}\cup W^-_{a+2,n} \subset W_{a+1,n+1}$.

We now fix $k\ge 1$ and $\kappa \in (\kappa^c_{\rho}(k),1)$, and 
for $x_0 \in W_{a,n}$, we introduce the events :
\begin{eqnarray*}
{\mathcal G}^+_{a,n}(x_0) & = & \left\{
\begin{array}{c}
\hbox{There exist distinct } x_1,\dots,x_k \in \chi_1 \cap  W^+_{a,n} \hbox{ and } x_{k+1} \in \chi_{\rho} \cap W^+_{a,n} \\
\hbox{such that } x_0, x_1, \dots, x_{k+1} \hbox{ is a path}
\end{array}
\right\}, \\
{\mathcal G}^-_{a,n}(x_0) & = & \left\{
\begin{array}{c}
\hbox{There exist distinct } x_1,\dots,x_k \in \chi_1 \cap  W^-_{a,n} \hbox{ and } x_{k+1} \in \chi_{\rho} \cap W^-_{a,n} \\
\hbox{such that } x_0, x_1, \dots, x_{k+1} \hbox{ is a path}
\end{array}
\right\}.
\end{eqnarray*}
Note that ${\mathcal G}^+_{0,0}(x)$ is exactly the event ${\mathcal G}^+(x)$ introduced in \eqref{e:bonevnt}, and that the other events are obtained from this one by symmetry and/or translation.

Next we choose $p \in (\vec{p}_c(2), 1)$. With Proposition \ref{p:onestep}, and by translation and symmetry invariance, we know that for every large enough dimension $d$, for every $(a,n) \in S$, for every $x \in W_{a,n}$:
\begin{equation}\label{e:gp}
P({\mathcal G}^{\pm}_{a,n}(x))\ge p.
\end{equation}
We fix then a dimension $d$ large enough to satisfy \eqref{e:gp}.
We can now construct the random states, open or closed, of the edges of our oriented graph. We denote by $\infty$ a virtual site.

\prg{Definition of the site on level $0$}
Almost surely, $\chi_\rho \cap W_{0,0} \neq \varnothing$. 
We take then some $x(0,0) \in \chi_\rho \cap W_{0,0}$. 

\prg{Definition of the edges between levels $n$ and $n+1$} Fix $n \ge 0$ and assume we built a site $x(a,n) \in W_{a,n}\cup \{\infty\}$ for every $a$ such that $(a,n) \in S$.  Consider $(a,n) \in S$ :
\begin{itemize}
\item If $x(a,n)=\infty$ :  
we decide that each of the two edges starting from $(a,n)$ is open with probability $p$ and closed with probability $1-p$, independently of everything else; we set $z^-(a,n)=z^+(a,n)=\infty$.
\item Otherwise : 
\begin{itemize}
\item Edge to the left-hand side : 
\begin{itemize}
\item if the event ${\mathcal G}^-_{a,n}(x(a,n))$ occurs : 
we take for $z^-(a,n)$ some point $x_{k+1}\in {W}^-_{a,n} \subset {W}_{a-1,n+1}$ given by the occurrence of the event, and we open the edge from $(a,n)$ to $(a-1,n+1)$  ; 
\item otherwise : we set $z^-(a,n)=\infty$ and we close the edge from $(a,n)$ to $(a-1,n+1)$. 
\end{itemize}
\item Edge to the right-hand side : 
\begin{itemize}
\item if the event ${\mathcal G}^+_{a,n}(x(a,n))$ occurs : 
we take for $z^+(a,n)$ some point $x_{k+1}\in {W}^+_{a,n} \subset {W}_{a+1,n+1}$ given by the occurrence of the event, and we open the edge from $(a,n)$ to $(a+1,n+1)$  ; 
\item otherwise : we set $z^+(a,n)=\infty$ and we close the edge from $(a,n)$ to $(a+1,n+1)$. 
\end{itemize}
\end{itemize}
\end{itemize}
For $(a,n)$ outside $S$, we set $z^{\pm}(a,n)=\infty$.

\prg{Definition of the sites at level $n+1$} Fix $n \ge 0$ and assume we determined the state of every edge between levels $n$ and $n+1$. Consider $(a,n+1) \in S$ :
\begin{itemize}
\item If $z^+(a-1,n) \neq \infty$ : set $x(a,n+1)=z^+(a-1,n)\in {W}_{a,n+1}$.
\item Otherwise :
\begin{itemize}
\item if $z^-(a+1,n) \neq \infty$ : set $x(a,n+1)=z^-(a+1,n) \in {W}_{a,n+1}$,
\item otherwise : set $x(a,n+1)=\infty$.
\end{itemize}
\end{itemize}

\bigskip
Assume that there exists an open path of length $n$ starting from the origin in this oriented percolation : we can check that the leftmost open path of length $n$ starting from the origin gives a path in the two-type Boolean model with $n$ alternating sequences of  $k$ balls with radius $1$  and one ball with radius $\rho$. Thus, percolation in this oriented percolation model implies percolation by $k$-alternation in the two-type Boolean model. Let us check that percolation occurs indeed with positive probability.

For every $n$, denote by ${\mathcal F}_n$ the $\sigma$-field generated by the restrictions of the Poisson point processes $\chi_1$ and $\chi_{\rho}$ to the set
$$
d^{-1/2} \left( \R \times (-\infty,n) \times \R^{d-2} \right).
$$
By definition of the events $\mathcal G$ -- remember that the  $(W_{a,n})_{(a,n) \in S}$ are disjoint -- and by \eqref{e:gp}, the states of the different edges between levels $n$ and $n+1$ are independent conditionally to ${\mathcal F}_n$. Moreover, conditionally to ${\mathcal F}_n$, each edge between levels $n$ and $n+1$ has a probability at least $p$ to be open. 
Therefore, the oriented percolation model we built stochastically dominates Bernoulli oriented percolation with parameter $p$. As $p>\vec{p}_c(\Z^2)$, with positive probability, there exists an infinite open path in the oriented percolation model we built; this ends the proof of Proposition \ref{p:supercritical}.
\finpreuve

\subsection{Proof of Theorem \ref{t:2}}

We first prove how Propositions \ref{p:subcritical} and \ref{p:supercritical} give 
\eqref{e:2} when $a=1$, $b>1$ and $\alpha=\beta=1$, and then we see how we can deduce the general case by scaling and coupling.

\subsubsection*{When $a=1$, $b>1$ and $\alpha=\beta=1$.} 
Set $\rho=b$. In this case, $\nu=\delta_1 + \delta_{\rho}$, so 
$\displaystyle \nu_d= \delta_1+\frac1{\rho^d}\delta_{\rho}$.

Note then that the two-type Boolean model $\Sigma$ introduced in Subsection \ref{s:ideas} and whose intensities depend on $\kappa \in (0,1)$ coincides with
the Boolean model directed by the measure 
$$
\frac{\kappa^d}{v_d 2^d} \nu_d
$$
as defined in the introduction.

If $\kappa<\kappa^c_\rho$ then, by Proposition \ref{p:subcritical}, there is no percolation for $d$ large enough.
Therefore, for any such $\kappa$ and for any large enough $d$ we have:
$$
\lambda^c_d(\nu_d) \ge \frac{\kappa^d}{v_d 2^d} \quad \text{and then} \quad \widetilde{\lambda}^c_d(\nu_d) = \lambda_d^c(\nu_d) v_d 2^d\int r^d \nu_d(dr)\ge 2\kappa^d .
$$
Letting $d$ goes to $+\infty$ and then $\kappa$ goes to $\kappa^c_\rho$, we then obtain
\begin{equation}
\label{e:liminf1}
\liminf_{d \to +\infty}\frac1d \ln \left(\lambda^c_d(\nu_d)\right)  \ge \ln \left( \kappa^c_\rho\right) .
\end{equation}

As $\kappa^c_\rho<1$ by Lemma \ref{l:kc2}, choose now $\kappa$ such that $\kappa^c_\rho<\kappa<1$. 
Then, there exists $k \ge 1$ such that $\kappa^c_{\rho}(k)<\kappa$.
Therefore, by Proposition \ref{p:supercritical}, there is percolation for $d$ large enough in $\Sigma$; by coupling, this remains true for 
larger $\kappa$.
Therefore, for any $\kappa>\kappa^c_\rho$ and for any large enough $d$ we have, as before:
$$
\lambda^c_d(\nu_d) \le \frac{\kappa^d}{v_d 2^d} \quad \text{and then} \quad\widetilde{\lambda}^c_d(\nu_d) \le 2\kappa^d.
$$
Letting $d$ goes to $+\infty$ and then $\kappa$ goes to $\kappa^c_\rho$, we then obtain
\begin{equation}
\label{e:limsup1}
\limsup_{d \to +\infty}\frac1d \ln \left(\lambda^c_d(\nu_d)\right)  \le \ln \left( \kappa^c_\rho\right) .
\end{equation}
Bringing \eqref{e:liminf1} and \eqref{e:limsup1} together, we get \eqref{e:2} when $a=1$, $b=\rho>1$ and $\alpha=\beta=1$.

\subsubsection*{When $b>a>0$ and $\alpha=\beta=1$.}
Set $\rho=b/a$. Here, $\nu=\delta_a+\delta_b$; set $\mu=\delta_1+\delta_\rho$. With the notation of the introduction, 
$$\nu_d=\frac1{a^d}(\delta_a+\frac1{\rho^d}\delta_b)=\frac1{a^d} H^a.(\delta_1+\frac1{\rho^d}\delta_\rho)=\frac1{a^d} H^a\mu_d.$$
By the scaling relations \eqref{e:scaling} and \eqref{e:scaling2}, we obtain
$$
\widetilde{\lambda}^c_d(\nu_d)=\widetilde{\lambda}^c_d(\mu_d).
$$
The result when $b>a>0$ and $\alpha=\beta=1$ follows then from the previous case.

\subsubsection*{When $b>a>0$ and $\alpha,\beta>0$.}
Here $\nu = \alpha\delta_a+\beta\delta_b$. Set $\mu=\delta_a+\delta_b$, $m=\min(\alpha,\beta)$ and $M=\max(\alpha,\beta)$. Then
$
m\mu_d \le \nu_d \le M\mu_d 
$
and so
\begin{eqnarray*}
m \int r^d d\mu_d(r) \le & \int r^d d\nu_d(r) & \le M \int r^d d\mu_d(r),\\
\frac1M\lambda_d^c(\mu_d) = \lambda_d^c(M\mu_d ) \le & \lambda_d^c(\nu_d) & \le \lambda_d^c(m\mu_d ) = \frac1m\lambda_d^c(\mu_d ).
\end{eqnarray*}
The two previous inequalities give:
$$
\frac{m}M \widetilde{\lambda}_d^c(\mu_d)
\le
\widetilde{\lambda}_d^c(\nu_d)
\le
\frac{M}m \widetilde{\lambda}_d^c(\mu_d),
$$
and the theorem follows from the previous case. \finpreuve

\section{Proof of Theorem \ref{t:general}}
\label{s:t:general}

Theorem \ref{t:general} follows from Theorem \ref{t:2} by coupling and scaling.
By assumption, $\mu$ is a measure on $(0,+\infty)$ whose support is not a singleton.
We can therefore choose $b'>a'>0$ in the support, set $\rho=b'/a'$  and then take a small enough $\epsilon>0$ such that
\begin{eqnarray*}
a'(1+\epsilon)<b'(1-\epsilon), && \mu([a'(1-\epsilon),a'(1+\epsilon)]>0, \\ \mu([b'(1-\epsilon),b'(1+\epsilon)]>0, &&
(1+\epsilon)(1-\epsilon)^{-1}\kappa^c_{\rho}<1.
\end{eqnarray*}
Set $a=a'(1-\epsilon), b=b'(1-\epsilon)$ and $\tau=(1+\epsilon)(1-\epsilon)^{-1}>1$.
We have
$$
a\tau < b, \; \mu([a,a\tau])>0, \; \mu([b,b\tau])>0 \text{ and } \tau \kappa^c_{\rho}<1.
$$
Set 
$
\nu=\mu([a,a\tau]) \delta_a + \mu([b,b\tau]) \delta_b
$
and
$
S=[a,a\tau] \cup [b,b\tau]
$.
For all $d \ge 1$ we have
\begin{eqnarray*}
\tau^{-d}\nu_d(\{a\}) 
 & = & \mu([a,a\tau]) (a\tau)^{-d} 
  \le  \int_{[a,a\tau]} r^{-d} \mu(dr) 
  =  \mu_d([a,a\tau])
\end{eqnarray*}
and, similarly,
$
\tau^{-d}\nu_d(\{b\}) \le  \mu_d([b,b\tau])
$.
By coupling, this implies that
$
\lambda_d^c(1_S\mu_d) \le \lambda_d^c(\tau^{-d}\nu_d)
$, and then that
$$
\lambda_d^c(\mu_d) \le \lambda_d^c(1_S\mu_d) \le \lambda_d^c(\tau^{-d}\nu_d) = \tau^d \lambda_d^c(\nu_d).
$$
But
$
\widetilde{\lambda}_d^c(\mu_d) = \lambda_d^c(\mu_d) 2^dv_d\mu((0,+\infty)) $
 and, similarly,
$\widetilde{\lambda}_d^c(\nu_d) = \lambda_d^c(\nu) 2^dv_d\nu((0,+\infty))$, which leads to
$$
\widetilde{\lambda}_d^c(\mu_d) \le \tau^d\frac{\mu((0,+\infty))}{\nu((0,+\infty))} \widetilde{\lambda}_d^c(\nu_d).
$$
But by Theorem \ref{t:2} we have
$$
\lim_{d\to+\infty} \frac1d\ln\left(\widetilde{\lambda}_d^c(\nu_d)\right) = \ln(\kappa^c_{\rho}), \text{ and then }
\limsup_{d\to+\infty} \frac1d\ln\left(\widetilde{\lambda}_d^c(\mu_d)\right) \le \ln(\tau\kappa^c_{\rho})<0,
$$
which ends ce proof.
\finpreuve

Note that as a byproduct of the proof, we obtain the following upper bound : $$
\limsup_{d\to+\infty} \frac1d\ln\left(\widetilde{\lambda}_d^c(\mu_d)\right) \le \inf_{0<a<b<+\infty,  \; a,b \in \text{Supp}(\mu)} \ln(\kappa^c_{b/a}).
$$


\bibliographystyle{plain}
\bibliography{biblio.bib}

\end{document}